\documentclass{article}

\usepackage{graphicx}
\usepackage{amsfonts}
\usepackage[T1]{fontenc}
\usepackage[french]{babel}
\usepackage{amssymb}
\usepackage{mathrsfs}
\usepackage{stmaryrd}
\usepackage{url}
\usepackage{amsmath}
\usepackage{wasysym}
\usepackage[all]{xy}
\usepackage[vcentermath]{youngtab}
\usepackage{tikz}
\usetikzlibrary{shapes.geometric}
\usetikzlibrary{calc}

\usepackage{caption}
\usepackage{subcaption}
\usepackage{hyperref}

\newtheorem{theorem}{Theorem}[section]
\newtheorem{lemma}[theorem]{Lemma}
\newtheorem{proposition}[theorem]{Proposition}

\newtheorem{conjecture}[theorem]{Conjecture}
\newtheorem{definition}{Definition}

\newenvironment{proof}[1][Proof]{\begin{trivlist}
\item[\hskip \labelsep {\bfseries #1}]}{\end{trivlist}}

\newenvironment{remark}[1][Remark]{\begin{trivlist}
\item[\hskip \labelsep {\bfseries #1}]}{\end{trivlist}}

\newenvironment{example}[1][Example]{\begin{trivlist}
\item[\hskip \labelsep {\bfseries #1}]}{\end{trivlist}}

    \newcommand\quotient[2]{
        \mathchoice
            {
                \text{\raise1ex\hbox{$#1$}\Big/\lower1ex\hbox{$#2$}}%
            }
            {
                #1\,/\,#2
            }
            {
                #1\,/\,#2
            }
            {
                #1\,/\,#2
            }
    }

\addto{\captionsfrench}{}

\def\Ueressl2{U_\varepsilon^{\mathrm{res}}(\widehat{\mathfrak{sl}_2})}
\def\Cres{\mathcal{C}^{}_\varepsilon}

\def\CZ{\mathcal{C}_{q^\mathbb{Z}}}
\def\CZres{\mathcal{C}_{\varepsilon^\mathbb{Z}}}

\def\Rep{\mathrm{Rep}\:}
\def\Uqlg{U_q(L\mathfrak{g})}
\def\Uqreslg{U_q^{\mathrm{res}}(L\mathfrak{g})}
\def\Uereslg{U_\varepsilon^{\mathrm{res}}(L\mathfrak{g})}

\def\Ueereslg{U_{\varepsilon^*}^{\mathrm{res}}(L\mathfrak{g})}
\def\Ueereslsl2{U_{\varepsilon^*}^{\mathrm{res}}(L\mathfrak{sl}_2)}
\def\Uereslsl2{U_{\varepsilon}^{\mathrm{res}}(L\mathfrak{sl}_2)}
\def\Uesl3{U_{\varepsilon}^{\mathrm{res}}(L\mathfrak{sl}_3)}

\numberwithin{equation}{section}

\title{Quantum affine algebras at roots of unity and generalised cluster algebras}

\author{Anne-Sophie Gleitz}

\date{}

\begin{document}

\maketitle

\begin{abstract}
Let $\Uereslsl2$ be the restricted integral form  of the quantum loop algebra $U_q(L\mathfrak{sl}_2)$ specialised at a root of unity $\varepsilon$. We prove that the Grothendieck ring of a tensor subcategory of representations of $\Uereslsl2$ is a generalised cluster algebra of type $C_{l-1}$, where $l$ is the order of $\varepsilon^2$. Moreover, we show that the classes of simple objects in the Grothendieck ring essentially coincide with the cluster monomials. We also state a conjecture for $\Uesl3$, and we prove it for $l=2$.
\end{abstract}

\section{Introduction}\label{intro}

Cluster algebras have been introduced in 2001 by Fomin and Zelevinski \cite{FZ1}. These rings have special generators, called \emph{cluster variables}. For every cluster $\mathbf{x}$, and every cluster variable $x\in\mathbf{x}$, there is a unique cluster $\left( \mathbf{x}\setminus \{x\}\right)\cup\{x'\}$, and an exchange relation
\begin{equation}\label{clu0}
xx' = m_+ + m_-\end{equation}
where $m_\pm$ are exchange monomials in $\mathbf{x}\setminus\{x\}$.
 Fomin and Zelevinsky \cite{FZ2} have proved a classification theorem for cluster algebras with finitely many clusters (also called of finite type), in terms of Cartan matrices.

We are interested in \emph{generalised cluster algebras}, introduced by Shapiro and Chekhov in 2011 \cite{CS}. The difference with standard cluster algebras resides in the exchange relations, whose right-hand side can include polynomials with more than two terms, unlike (\ref{clu0}). Otherwise, finite type classification and combinatorial behaviour stay the same \cite{CS}. 

We focus on a generalised cluster algebra $\mathcal{A}_n$ of Cartan type $C_n$, with a particular choice of coefficients, and describe its inner combinatorics. In particular, we describe several $\mathbb{Z}$-bases of $\mathcal{A}_n$.

On the other hand, the theory of finite-dimensional representations of the quantum loop algebra $U_q(L{\mathfrak{g}})$ for $q\in\mathbb{C}^*$ not a root of unity is well established. In this paper, we are interested in the case where $q=\varepsilon$ is a root of unity. The algebra $\Uqlg$ is then replaced by the restricted integral form $\Uereslg$, introduced and studied by Chari and Pressley \cite{CP2}, and later by Frenkel and Mukhin \cite{FM}.

In the spirit of Hernandez and Leclerc's papers \cite{HL} and \cite{HL2}, we consider a certain tensor category $\CZres$ of finite-dimensional $\Uereslg$-modules, and we show that when $\mathfrak{g}=\mathfrak{sl}_2$, the Grothendieck ring of $\CZres$ is isomorphic to $\mathcal{A}_{l-1}$ (see Theorem \ref{conjA1}), where $l$ is the order of $\varepsilon^2$. Moreover, under this isomorphism, the basis of classes of simple objects of $\CZres$ coincides with the basis of (generalised) cluster monomials, multiplied by Tchebychev polynomials in the single generator of the coefficient ring. 
This is proved by combining tools from the theory of generalised cluster algebras (see Section \ref{s21}), and from the representation theory of $\Uereslsl2$ (see Section \ref{s22}).

For $\mathfrak{g}=\mathfrak{sl}_3$ and $l=2$, we prove a similar result, where $\mathcal{A}_{l-1}$ is replaced by a generalised cluster algebra of type $G_2$. Extensive computations with Maple allow us to formulate a conjecture for $\mathfrak{g}=\mathfrak{sl}_3$ and $l>2$. However, the generalised cluster algebras occurring in this conjecture are of infinite type, and we still lack the proper tools to prove it.

\subsection*{Acknowledgements}
The author would like to thank B. Leclerc for his invaluable advice and insight throughout this work.

\section{Cluster algebras}\label{s21}

We are interested in a structure that generalises the notion of cluster algebras, defined by Shapiro and Chekhov in \cite{CS}.

\subsection{Generalised cluster algebras }\label{s212}

We recall, following \cite{CS}, the definition and the main structural properties of generalised cluster algebras, see also \cite{N}.

For a fixed integer $n\in\mathbb{N}^*$, let $B=(b_{ij})\in\nolinebreak\mathcal{M}_n(\mathbb{Z})$ be a skew-symmetrisable matrix, i.e. such that there exists an integer diagonal matrix $\tilde{D}=\mathrm{diag}(\tilde{d}_1\dots\tilde{d}_n)$ such that $\tilde{D}B$ is skew-symmetric. 

Suppose that for each index $k\in\llbracket 1,n\rrbracket$, there is an integer $d_k\in\mathbb{N}$ that divides all coefficients $b_{ j k}$ in the $k$-th column. Introduce the notation
\begin{equation}\beta_{jk} := \displaystyle \frac{b_{jk}}{d_k} \in\mathbb{Z}.\end{equation}

Let $(\mathbb{P},\cdot,\oplus)$ be a commutative semifield, called the \emph{coefficient group}.  For example, one can take for $\mathbb P$ the tropical semifield $\mathrm{Trop}(\lambda_1,\dots,\lambda_n)$ generated by some indeterminates $\lambda_1,\dots,\lambda_n$. This is by definition the set of Laurent monomials in the $\lambda_i$'s, with ordinary multiplication and tropical addition
\begin{equation*}
\left(\displaystyle\prod_i \lambda_i^{a_i}\right) \oplus \left(\displaystyle\prod_i \lambda_i^{b_i}\right) = \left(\displaystyle\prod_i \lambda_i^{\min(a_i,b_i)}\right).
\end{equation*}
Let $\mathcal{F}=\mathbb{ZP}(t_1,\dots,t_n)$ be the ambient field of rational functions in $n$ independent variables, where $\mathbb{ZP}$ is the integer group ring of $\mathbb{P}$.

For a collection of variables $\mathbf{p}_i=(p_{i,0}, p_{i,1},\dots, p_{i,d_i})\in\mathbb{P}^{d_i+1}\quad(i\in\llbracket 1,n\rrbracket)$, define the corresponding homogeneous \emph{exchange polynomial}
\begin{equation}\theta_i\lbrack \mathbf{p}_i\rbrack(u,v) := \displaystyle \sum_{r=0}^{d_i} p_{i,r} u^r v^{d_i-r}\in\mathbb{ZP}\lbrack u,v\rbrack.\end{equation}

\begin{definition}
A \emph{generalised seed} is a triple $(\mathbf{x},\bar{\mathbf{p}},B)$ where
\begin{enumerate}
\item[(i)] the tuple $\mathbf{x}=\{x_1,\dots,x_n\}$, called a \emph{cluster}, is a collection of algebraically independent elements of $\mathcal F$, called \emph{cluster variables}, which generate $\mathcal F$ over $\mathrm{Frac}\:\mathbb{ZP}$;
\item[(ii)] the matrix $B=(b_{ij})\in\mathcal{M}_n(\mathbb{Z})$, called the \emph{exchange matrix}, is skew-symmetrisable;
\item[(iii)]  $\bar{\mathbf{p}}=(\mathbf{p}_1,\dots,\mathbf{p}_n)$ is a \emph{coefficient tuple},  where for each $i\in\llbracket 1,n\rrbracket$, the tuples $\mathbf{p}_i=(p_{i,0},p_{i,1},\dots,p_{i,d_i})\in\mathbb{P}^{d_i+1}$  are the coefficients of the $i$-th exchange polynomial $\theta_i$.

\end{enumerate} 
\end{definition}

The triple  $(\mathbf{x},\{\theta_1,\dots,\theta_n\},B)$ is also called a generalised seed.

\begin{definition}\label{genmut}
The \emph{generalised mutation in direction $k\in\llbracket 1,n\rrbracket$}, is the operation that transforms a generalised seed $(\mathbf{x},\bar{\mathbf{p}},B) $ into another generalised seed  $\mu_k(\mathbf{x},\bar{\mathbf{p}},B):=(\mathbf{x}',\bar{\mathbf{p}'},B')$
given by
\begin{enumerate}
\item[(i)]\mbox{\emph{matrix mutation}}: the matrix $B'=(b'_{ij})$ is defined by
\begin{equation}b'_{ij} =\left\{\begin{array}{ll}
-b_{i j} &\mbox{if } i=k \mbox{ or } j=k\\
b_{ij} + \displaystyle\frac{1}{2}\left( |b_{ik}| b_{k j} + b_{ik} |b_{kj}|\right) &\mbox{otherwise} 
\end{array}  \right.\end{equation}
\item [(ii)] \mbox{\emph{cluster mutation:}}
\begin{equation}\left\{\begin{array}{ll}
x'_i = x_i &\mbox{if } i\neq k\\
x_k x'_k = \theta_k\lbrack \mathbf{p}_k\rbrack(u_k^+,u_k^-)
\end{array}\right.\end{equation}
where we define
\begin{equation}u_k^+ := \displaystyle \prod_{j=1}^n x_j^{\lbrack \beta_{jk}\rbrack_+} \quad\mbox{and}\quad u_k^- := \displaystyle \prod_{j=1}^n x_j^{\lbrack -\beta_{jk}\rbrack_+} .\end{equation}
\item[(iii)]  \mbox{\emph{coefficient mutation:}} 
\begin{equation}\left\{\begin{array}{ll}
p'_{k,r}=p_{k,d_k-r}\\\\
\displaystyle \frac{p'_{i,r}}{p'_{i,r-1}}&= \left\{ \begin{array}{ll}
(p_{k,d_k})^{\beta_{k i}} \displaystyle \frac{p_{i,r}}{p_{i,r-1}} &\mbox{if } i\neq k \mbox{ and } b_{ki}\geq 0\\\\
(p_{k,0})^{\beta_{k i}} \displaystyle \frac{p_{i,r}}{p_{i,r-1}} &\mbox{if } i\neq k \mbox{ and } b_{ki}\leq 0
\end{array} \right.
\end{array}\right.\end{equation}

\end{enumerate}
\end{definition}

For $r\in\llbracket 1,n\rrbracket$, write $\mu_r(B):=(b'_{ij})$. It follows easily from the definition of matrix mutation that for each $k\in\llbracket 1,n\rrbracket$,  the integer $d_k$ divides all coefficients in the $k$-th column of $\mu_r(B)$. Moreover, note that $\mu_r$ is an involution. We say that two generalised seeds are \emph{mutation-equivalent} if one can be obtained from the other by performing a finite sequence of mutations. 

Observe that if $d_i=1$ for all $i$, then the exchange polynomials are of the form $\theta_i(u,v)=p_{i,0}u+p_{i,1}v$. We then recover the ordinary notions of seed and seed mutation from \cite{FZ1} and \cite{FZ2} by setting $p_i^+=p_{i,1}$ and $p_i^-=p_{i,0}$.

\begin{definition}
The \emph{generalised cluster algebra} $\mathcal{A}(\bar{\mathbf{p}},B)=\mathcal{A}(\mathbf{x},\{\theta_1,\dots,\theta_n\},B)$ of \emph{rank} $n$, corresponding to the generalised seed $(\mathbf{x},\{\theta_1,\dots,\theta_n\},B)$, is the $\mathbb{ZP}$-subalgebra of $\mathcal{F}$
 generated by all cluster variables from all the seeds that are mutation-equivalent to the initial seed $(\mathbf{x},\{\theta_1,\dots,\theta_n\},B)$.
\end{definition}

We say that a generalised cluster algebra is of \emph{finite type} if it has finitely many cluster variables. 
The Laurent phenomenon from \cite{FZ1} remains true for generalised cluster algebras.

\begin{theorem}[{\cite[Theorem 2.5]{CS}}]\label{geneLau}
Every generalised cluster variable is a \\Laurent
polynomial in the initial cluster variables.
\end{theorem}

Two generalised cluster algebras $\mathcal{A}(\bar{\mathbf{p}},B)\subset\mathcal{F}$ and $\mathcal{A}(\bar{\mathbf{p}}',B')\subset\mathcal{F}'$ over the same semifield $\mathbb{P}$ are called \emph{strongly isomorphic} if there is a $\mathbb{ZP}$-isomorphism $\mathcal{F}\rightarrow\mathcal{F}'$ that sends any generalised seed of $\mathcal{A}(\bar{\mathbf{p}},B)$ onto a generalised seed $\mathcal{A}(\bar{\mathbf{p}}',B')$.
 This induces a bijection between the sets of generalised seeds, as well as an algebra isomorphism $\mathcal{A}(\bar{\mathbf{p}},B)\cong\mathcal{A}(\bar{\mathbf{p}}',B').$
 Every generalised cluster algebra $\mathcal{A}(\bar{\mathbf{p}},B)$ over a semifield $\mathbb{P}$ belongs to a series $\mathcal{A}(-,B)$, consisting in all the generalised cluster algebras $\mathcal{A}(\bar{\mathbf{p}},B)$ where $B$ is fixed and $\bar{\mathbf{p}}$ may vary.
 We say that two series $\mathcal{A}(-,B)$ and $\mathcal{A}(-,B')$ are \emph{strongly isomorphic} if $B$ and $B'$ are mutation-equivalent, modulo simultaneous relabeling of  rows and columns.


Let $M=(m_{ij})\in\mathcal{M}_n(\mathbb{Z})$. The Cartan counterpart of $M$ is the generalised Cartan matrix $A=A(M)=(a_{ij})\in\mathcal{M}_n(\mathbb{Z})$ defined by
\begin{equation}a_{ij}=\left\{\begin{array}{ll}
2&\mbox{if }i=j\\
-|m_{ij}|&\mbox{if }i\neq j.
\end{array}              \right.\end{equation}


\begin{theorem}[{\cite[Theorem 2.7]{CS}}]\label{geneCla}
Generalised cluster algebras of finite type follow the same Cartan-Killing classification as standard cluster algebras. Namely, there is a canonical bijection between the Cartan matrices of finite type and the strong isomorphism classes of series of generalised cluster algebras of finite type. Under this bijection, a Cartan matrix $A$ of finite type corresponds to the series $\mathcal{A}(-,B)$, where $B$ is a skew-symmetrisable matrix such that $A(B)=A$. 
\end{theorem}

Finally, we recall that a \emph{cluster monomial} in $\mathcal A(\bar{\mathbf p},B)$ is a monomial in the cluster variables involving only variables belonging to a single cluster.

\subsection{A generalised cluster algebra of type \texorpdfstring{$C_n$}{Cn}}\label{s213}

\subsubsection{Combinatorics and exchange relations}\label{s231}

The \emph{exchange graph} of a cluster algebra is the graph whose vertices are the clusters, and two clusters are linked by an edge if they can be obtained from each other by one mutation.

We know from \cite[Section 12.3]{FZ2} that for a cluster algebra of type $C_{n}$, the exchange graph is isomorphic to the $n$-dimensional cyclohedron. It has a nice description in terms of triangulations of a regular $(2n+2)$-gon $\mathbf{P}_{2n+2}$.

More precisely, each cluster variable can be associated with either a centrally symmetric pair of diagonals, or a diameter. Under this bijection, each vertex of the exchange graph corresponds to a centrally symmetric triangulation, and two such triangulations are linked by an edge if they can be obtained from each other either by a flip involving two diameters, or by a pair of centrally symmetric flips. Note that each centrally symmetric triangulation contains a unique diameter.

For a standard cluster algebra, the exchange relations correspond to Ptolemy relations in the appropriate quadrilaterals. For a generalised cluster algebra, certain formulas are slightly more complicated, see Proposition \ref{pCn} below.

Let us identify the set $\Sigma$ of vertices of $\mathbf{P}_{2n+2}$ with the cyclic group  $$\mathbb{Z}/(2n+2)\mathbb{Z}\cong 2 \mathbb{Z}/(4n+4)\mathbb{Z},$$ by labelling the vertices clockwise: 0,2,4,$\dots$,$2n$,\\$2n+2$,$2n+4$,$\dots$, $4n+2$, with the natural additive law induced by the cyclic group. We rename half of the vertices in the following way: for each $k\in\llbracket 0,n\rrbracket$, write
\begin{equation}(2n+2)+2k := \overline{2k}.\end{equation}
In particular, $2n+2=\bar{0}$ and $\overline{2n}+2=0$.
 This makes it easier to identify centrally symmetric pairs of diagonals. It might seem odd to use "$2k$" instead of just "$k$", but this notation will turn out to be the most natural one for Section \ref{s3}.

Let $\mathscr{C}$ be the circle in which $\mathbf{P}_{2n+2}$ is inscribed, and let $\Theta$ be the central symmetry around the center of $\mathscr{C}$.
 Consider a pair $\{\lbrack a,b \rbrack, \lbrack \bar{a},\bar{b}\rbrack\}$ of centrally symmetric diagonals.
We may choose $\lbrack a,b\rbrack$ to represent the $\Theta$-orbit of this pair. The segment  $\lbrack a,b\rbrack$ divides the circle $\mathscr{C}$ into two arcs. The $\Theta$-orbits of the vertices of $\mathbf{P}_{2n+2}$ that lie on the smallest arc form a set denoted by $\mathcal{O}_{ab}$. For example, if $a<b\in\llbracket 0,2n\rrbracket$, the set $\mathcal{O}_{ab}$ consists of the $\Theta$-orbits of $a+2$, $a+4$,$\dots$, $b-2$. 
In general, we have $\mathcal{O}_{ab}=\mathcal{O}_{ba}=\mathcal{O}_{\bar{a}\bar{b}}=\mathcal{O}_{\bar{b}\bar{a}}.$

\begin{example}
In type $C_3$, the vertices of the regular octagon $\mathbf{P}_8$ will be numbered as in Figure \ref{fig:octa}. 
 For example, the pair of centrally symmetric diagonals $\{\lbrack 2,\bar{0}\rbrack,\lbrack  \bar{2},0\rbrack\}$ corresponds to the set $\mathcal{O}_{2,\bar{0}}$, which consists of the $\Theta$-orbits $\{4,\overline{4}\}$ and $\{6,\overline{6}\}$.

\begin{figure}[!ht]
\centering
\begin{tikzpicture}
\newdimen\R
\R=1.1cm
\draw (0:\R)
\foreach [count=\n] \x in {45,90,...,360} {
        -- (\x:\R) }
 -- cycle (180:\R) node[left] {0}
-- cycle (270: \R) node[below] {$\bar{4}$}
-- cycle (315: \R) node[below right] {$\bar{2}$}
              --cycle (225: \R) node[below left] {$\bar{6}$}
              -- cycle (0:\R) node[ right] {$\bar{0}$}
              -- cycle (135:\R) node[above left] {2}
              -- cycle (90:\R) node[above] {4}
              -- cycle  (45:\R) node[above right] {6};
\end{tikzpicture}
\caption{The regular octagon $\mathbf{P}_8$}\label{fig:octa}
\end{figure}
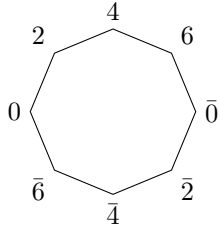
\end{example}

We label cluster variables by the corresponding $\Theta$-orbits of diagonals. Namely, if $b\neq \overline{a}$, the variable $x_{ab}$ corresponds to the pair of diagonals $\{\lbrack a,b\rbrack, \lbrack \bar{a},\bar{b}\rbrack\}$. If $b=\overline{a}$, the variable $x_{a\overline{a}}$ corresponds to the diameter $\{\lbrack a,\overline{a}\rbrack\}$. Thus we have \begin{equation}x_{ab}=x_{ba}=x_{\bar{a}\bar{b}}=x_{\bar{b}\bar{a}}.\end{equation}
By convention, if $a$ and $b$ are neighbours in $\mathbf{P}_{2n+2}$, we set $x_{ab}=1$. 
Note that each cluster variable $x_{ab}$ may also be labelled by the set $\mathcal{O}_{ab}$.

Theorem \ref{geneCla} allows us to use the same labeling system for a generalised cluster algebra of type $C_n$. In particular, mutations can be seen as flips between triangulations of $\mathbf{P}_{2n+2}$. 

The following example is a particular case of the more general Definition \ref{dCn}.

\begin{example} In type $C_3$, 
consider  the following initial seed, with coefficient group $\mathbb{P}=\mathrm{Trop}(\lambda)=\mathbb Z \lbrack \lambda^{\pm 1}\rbrack$:
\begin{equation}\Pi_0:=(\mathbf{x}^{(0)},\{\theta_1,\theta_2,\theta_3\},B)\end{equation}
where we set
\begin{equation}\mathbf{x}^0 = (x_{2\bar{6}},x_{4\bar{6}},x_{6\bar{6}}),\quad \left\{ \begin{array}{l}
\theta_1(u,v)=u+v\\
\theta_2(u,v)=u+v\\
\theta_3(u,v)=u^2 + \lambda u v + v^2
\end{array}\right.,B=\left( \begin{array}{ccc}
0&1&0\\-1&0&2\\0&-1&0
\end{array}\right).   \end{equation}
Because of the special choice of coefficients, the exchange polynomials remain unaffected by mutation. 
Mutating $\Pi_0$, we obtain twelve cluster variables, which can be organised in 20 clusters, as in Figure \ref{ex:c3}. 
This corresponds to the 3-dimensional cyclohedron whose vertices are the non-crossing centrally symmetric triangulations of the octagon (Figure \ref{fig:cyclo3d}).

\begin{figure}[!ht]
\hspace*{0.5cm}
$\xymatrix @R=-0.5pc @C=-0.4pc @M=0.0pc { 
&&&\mbox{$\begin{array}{c}x_{2\bar{6}},x_{4\bar{6}},\\x_{4\bar{4}}\end{array}$}\ar@{-}[rr] &&\mbox{$\begin{array}{c}x_{2\bar{6}},x_{2\bar{4}},\\x_{4\bar{4}}\end{array}$}\ar@{-}[rrd] &&&&&\\
&\mbox{$\begin{array}{c}x_{04},x_{4\bar{6}},\\x_{4\bar{4}}\end{array}$}\ar@{-}[rru]\ar@{-}[ld]&&&&&&\mbox{$\begin{array}{c}x_{0\bar{4}},\\x_{2\bar{4}},\\x_{4\bar{4}}\end{array}$}&&&&&&&&&&&&&&&&&&\\
\mbox{$\begin{array}{c}x_{04},x_{4\bar{6}},\\x_{6\bar{6}}\end{array}$}\ar@{-}[rr] & &\mbox{$\begin{array}{c}x_{2\bar{6}},\\x_{4\bar{6}},\\x_{6\bar{6}}\end{array}$}\ar@{-}[uur]\ar@{-}[dr]&&\mbox{$\begin{array}{c}x_{04},\\x_{06},\\x_{4\bar{4}}\end{array}$}\ar@{.}[ulll]\ar@{.}[dd]\ar@{.}[urrr]&&\mbox{$\begin{array}{c}x_{2\bar{6}},\\x_{2\bar{4}},\\x_{2\bar{2}}\end{array}$}\ar@{-}[uul]\ar@{-}[dl]\ar@{-}[rr] &&\mbox{$\begin{array}{c}x_{0\bar{4}},x_{2\bar{4}},\\x_{2\bar{2}}\end{array}$}\ar@{-}[ul]&&&&&&&&&&&&&&&&&&&\\
\mbox{$\begin{array}{c}x_{04},x_{06},\\x_{6\bar{6}}\end{array}$}\ar@{-}[ddr]\ar@{-}[u]\ar@{-}[dr]&&&\mbox{$\begin{array}{c}x_{2\bar{6}},\\x_{26},\\x_{6\bar{6}}\end{array}$}\ar@{-}[rr]\ar@{-}[dll]&&\mbox{$\begin{array}{c}x_{2\bar{6}},\\x_{26},\\x_{2\bar{2}}\end{array}$}\ar@{-}[drr]&&&\mbox{$\begin{array}{c}x_{0\bar{4}},x_{0\bar{2}},\\x_{2\bar{2}}\end{array}$}\ar@{-}[ddl]\ar@{-}[u]\ar@{-}[dl]&&&&&&&&&&&&&&&&&&&&&&&&&&&&&&&&&&&&&\\
&\mbox{$\begin{array}{c}x_{06},x_{26},\\x_{6\bar{6}}\end{array}$}\ar@{-}[ddrr]&&&\mbox{$\begin{array}{c}x_{04},\\x_{0\bar{4}},\\x_{0\bar{0}}\end{array}$}\ar@{.}[dlll]\ar@{.}[drrr]&&&\mbox{$\begin{array}{c}x_{26},x_{0\bar{2}},\\x_{2\bar{2}}\end{array}$}\ar@{-}[ddll]&&&&&&&&&&&&&&&&&&&&&&&\\
&\mbox{$\begin{array}{c}x_{04},x_{06},\\x_{0\bar{0}}\end{array}$}\ar@{-}[rrd]&&&&&&\mbox{$\begin{array}{c}x_{0\bar{4}},x_{0\bar{2}},\\x_{0\bar{0}}\end{array}$}&&&&&&&&&&&&&&
\\
&&&\mbox{$\begin{array}{c}x_{26},x_{06},\\x_{0\bar{0}}\end{array}$}\ar@{-}[rr]&&\mbox{$\begin{array}{c}x_{26},x_{0\bar{2}},\\x_{0\bar{0}}\end{array}$}\ar@{-}[rru]&&&&&&&&&&&&&&&&&&&
}$
\caption{The clusters in type $C_3$}\label{ex:c3}
 \end{figure}

\begin{figure}[!ht]
\centering
\includegraphics[bb=80 10 500 430, width=8.2cm, height = 8.2cm, scale = 1]{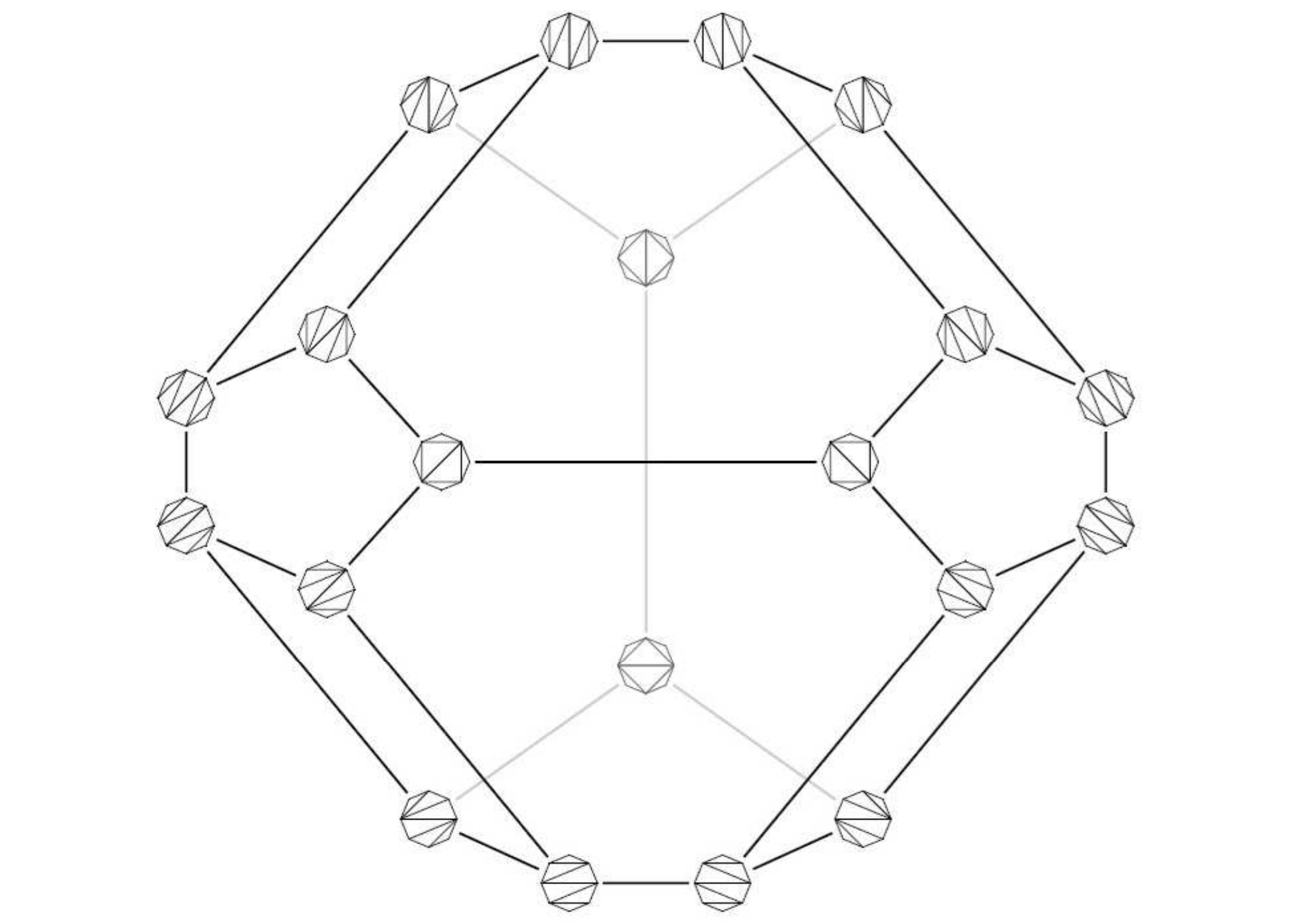}
\caption{3-dimensional cyclohedron: flips of the regular octagon (\cite[Figure  3.9]{FR2})}
\label{fig:cyclo3d}\end{figure}

\end{example}

\begin{definition}\label{dCn}
Let $\mathbb{P}=\mathrm{Trop}(\lambda)$. For an integer $n\in\mathbb{N},\:n\geq 2$, we denote by $\overline{\mathcal{A}}_n=\mathcal{A}(\mathbf{x},\{\theta_1^0,\dots,\theta_n^0\},B)$  the generalised cluster algebra defined by the initial seed 
\begin{equation}\theta_i^0(u,v)=u+v\:\:(i\in\llbracket 1,n-1\rrbracket), \qquad \theta_n^0(u,v)=u^2+\lambda uv+v^2,\end{equation}
and
\begin{equation}B:=\left(  \begin{array}{ccccccc}
0&1&0&0&0&\dots&0\\
-1&0&1&0&0&\dots&0\\
0&-1&0&1&0&\dots&0\\
0&0&-1&0&\ddots&\ddots&\vdots\\
\vdots&\vdots&\ddots&\ddots&\ddots&1&0\\
0&0&\dots&0&-1&0&2\\
0&0&0&\dots&0&-1&0
\end{array} \right).\end{equation}

\end{definition}

Thus $\overline{\mathcal{A}}_n$ is the $\mathbb Z\lbrack \lambda^{\pm 1}\rbrack$-subalgebra of $\mathcal F$ generated by the cluster variables. We will rather work with a variant of $\overline{\mathcal{A}}_n$, in which the coefficient $\lambda$ is not assumed to be invertible.

\begin{definition}\label{dAn}
Let $\mathcal A_n$ be the $\mathbb Z\lbrack\lambda\rbrack$-subalgebra of $\mathcal F$ generated by the cluster variables of $\overline{\mathcal{A}}_n$.
\end{definition}

As above, we label the cluster variables of $\overline{\mathcal{A}}_n$ (or $\mathcal{A}_n$) by $\Theta$-orbits of diagonals of $\mathbf{P}_{2n+2}$. The initial cluster variables corresponding to the initial seed are as follows (see Figure \ref{fig:initclu}):
\begin{equation}x_k:=x_{\overline{2n},2k}\quad(k\in\llbracket 1,n\rrbracket).\end{equation}

\bigskip
\bigskip
\bigskip
\bigskip
\begin{figure}[!ht]
\centering
\includegraphics[bb=150 15 400 265, width=3.5cm, height = 3.5cm, scale = 1]{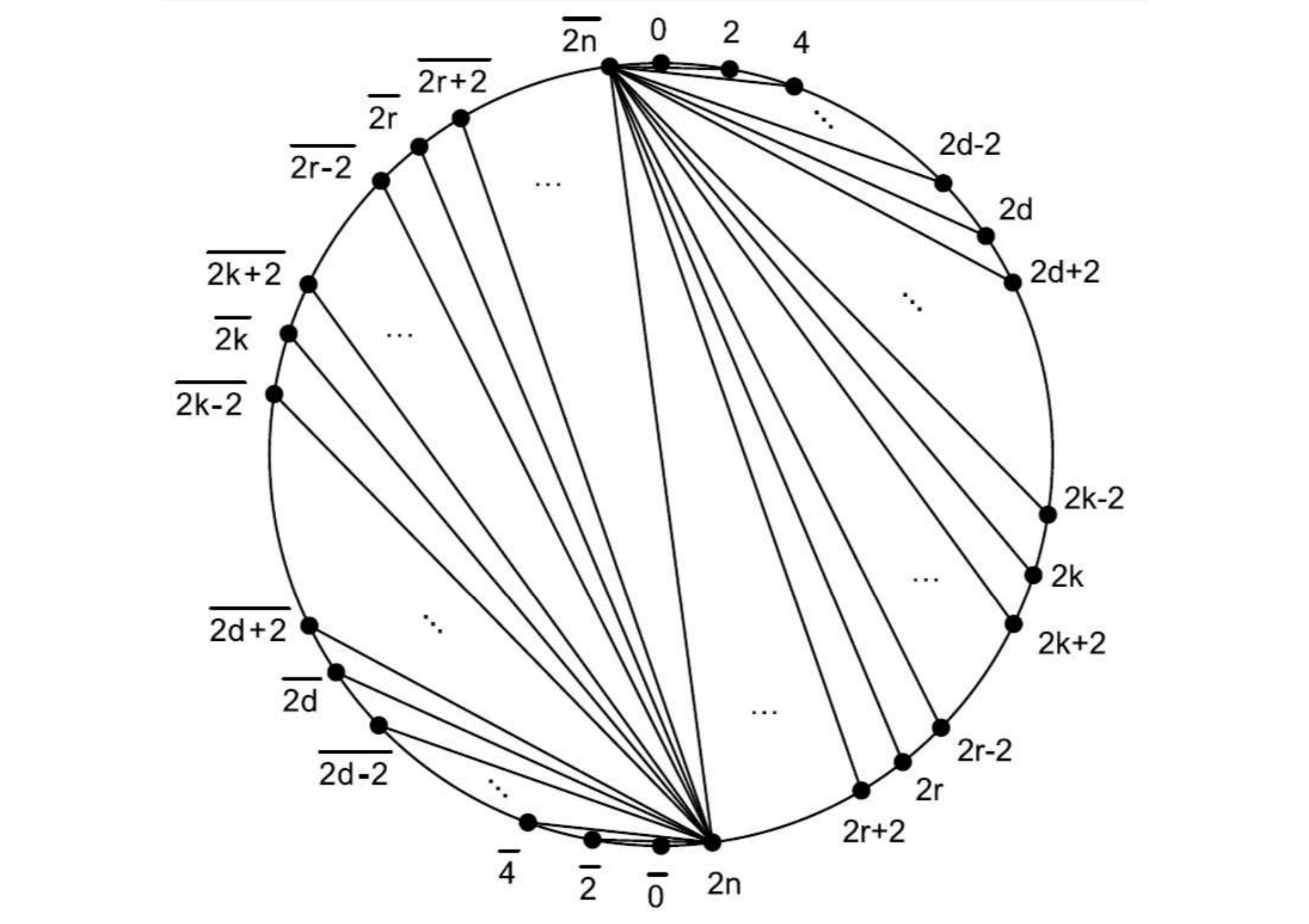}
\caption{The initial cluster of $\mathcal{A}_n$}
\label{fig:initclu}\end{figure}

\begin{remark}
\begin{enumerate}
\item It follows from Definition \ref{genmut}, that the exchange polynomials $\theta_i^0, \:i<n$, are not affected by mutation; moreover, they coincide with standard exchange relations.
\item There is exactly one variable of the form $x_{a,\bar{a}}$ in each cluster. Indeed, as noted above, any centrally symmetric triangulation contains exactly one diameter. 
Moreover, the exchange polynomial $\theta_n^0$ also remains unaffected by mutation. 
Therefore, mutating $x_{a,\overline{a}}$ will yield a variable of the form $x_{b,\overline{b}}$. 
This can also be seen in terms of triangulations: flipping a diameter while keeping a non-crossing, centrally-symmetric triangulation of $\mathbf{P}_{2n+2}$, gives another diameter (it is easy to see that the quadrilaterals in which diameters are flipped, are always rectangles). 
\end{enumerate}
\end{remark}

\begin{proposition}\label{pCn}
In the generalised cluster algebra $\mathcal{A}_n$, the following exchange relations between variables  $x_{ab}$ and $x_{cd}$ hold, up to rotation (i.e. index shifting):
\begin{enumerate}
\item If $a\neq\bar{b}$, $c\neq \bar{d}$, and the quadrilateron $\lbrack acbd\rbrack$ is contained in one half of the circle $\mathscr{C}$, we have a standard exchange relation of the form
\begin{equation}\label{gexc1}
x_{\overline{2n},2k+2} x_{2d-2,2r+2}= x_{\overline{2n},2r+2} x_{2d-2,2k+2} + x_{\overline{2n},2d-2} x_{2k+2,2r+2},
\end{equation}
which corresponds to the Ptolemy rule in the first diagram of Figure \ref{fig:diagex}.
\item if $a=\bar{b}$ and $c=\bar{d}$, we have a generalised exchange relation of the form
\begin{equation}\label{gexc2}
x_{\overline{2n},2n}x_{\overline{2k},2k} = x_{\overline{2n},2k}^2+x_{2k,2n}^2+\lambda x_{\overline{2n},2k}x_{2n,2k}.
\end{equation}
The monomials with coefficient 1 correspond to the Ptolemy rule in the second diagram of Figure \nolinebreak\ref{fig:diagex}.
\end{enumerate}

\end{proposition}

\bigskip
\bigskip
\bigskip
\begin{figure}[!ht]
\centering
\includegraphics[bb=230 00 530 300, width=4cm, height = 4cm, scale = 1]{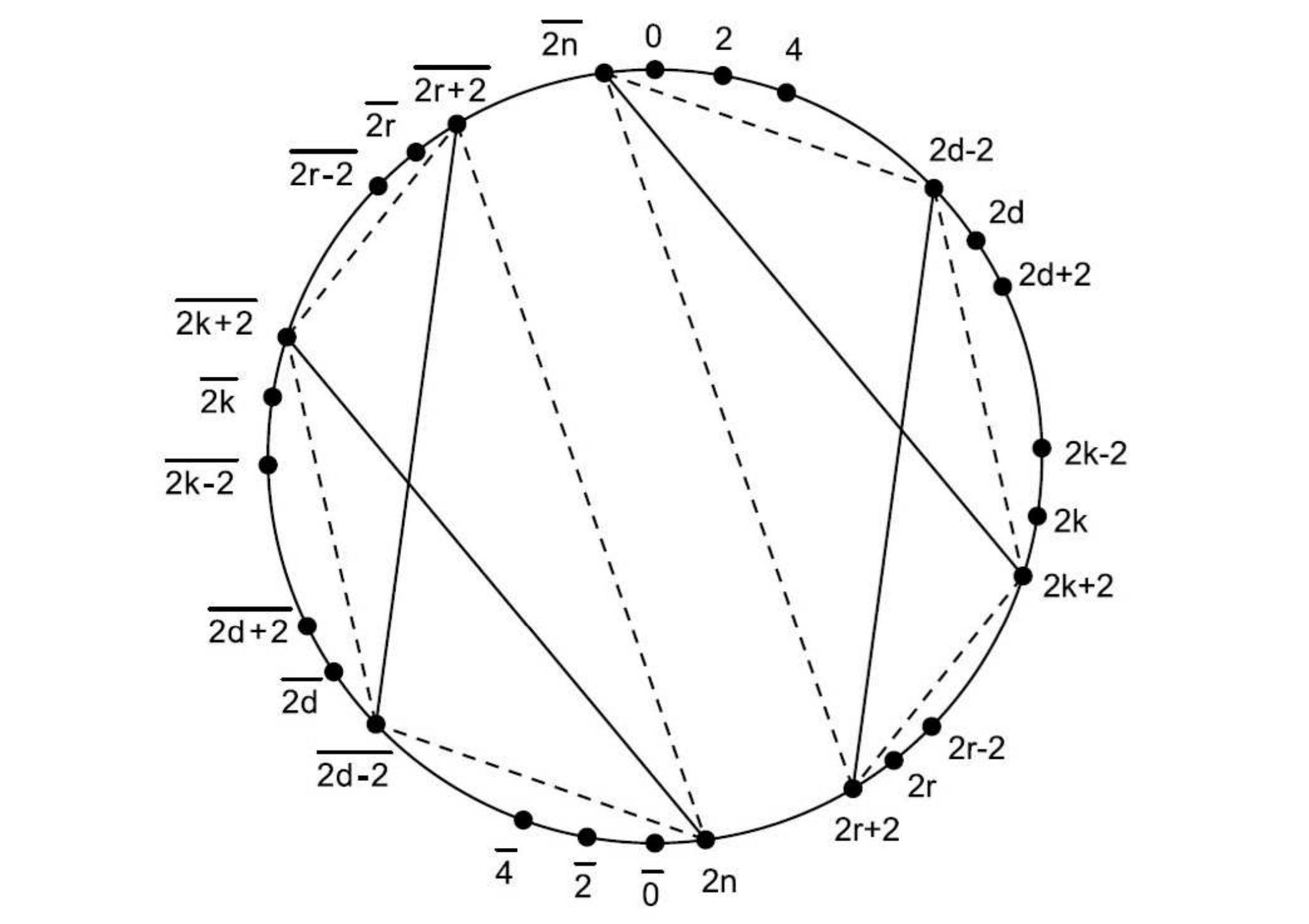}
\includegraphics[bb=00 00 295 295, width=4cm, height = 4cm, scale = 1]{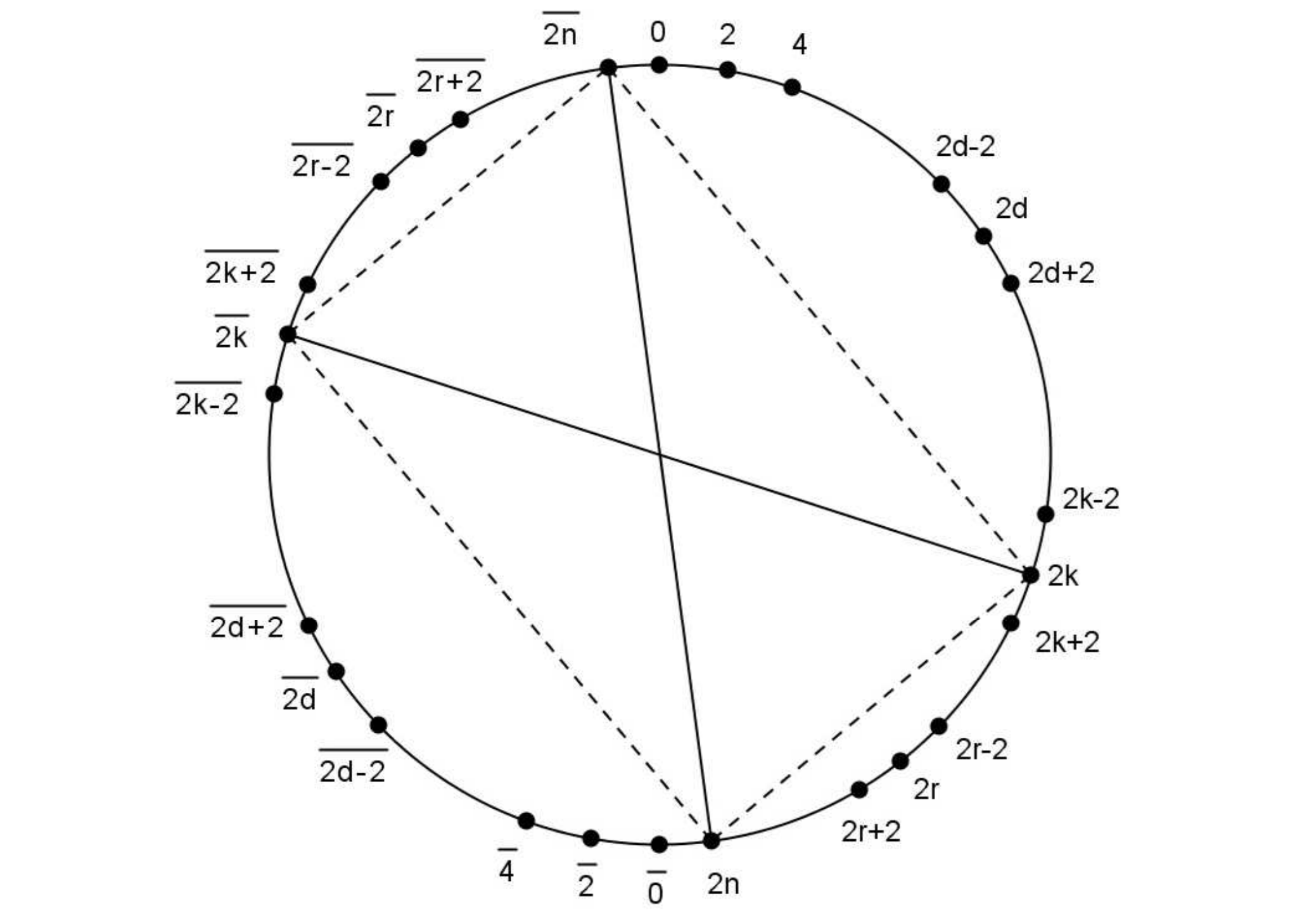}
\caption{Generalised exchange relations for $\mathcal{A}_n$}
\label{fig:diagex}\end{figure}

\begin{proof}
The identity (\ref{gexc1}), for $d=r=k+1$, is by definition true in the initial seed. Moreover, the exchange polynomials $\theta_1^0,\dots,\theta_{n-1}^0$ are exactly the ones that appear in a standard cluster algebra of type $C$, and they are unaffected by mutation: indeed, only the monomials $u_k^\pm$ change, in accordance with the mutations of the exchange matrix. Therefore, cluster variables that do not correspond to diameters behave the same way as in a standard cluster algebra of type $C$, as described in \cite{FZ2}. The proof for the first case is thus similar to those found in \cite{CFZ} and \cite{FZ2}.

The second equation (\ref{gexc2}), for $k=n-1$, is also by definition true in the initial seed. 
 Since every cluster contains exactly one variable of the form $x_{a,\overline{a}}$, and any mutation of a variable $x_{a,\overline{a}}$ yields a variable corresponding to another diameter, we can deduce from the initial cluster that all variables $x_{a,\overline{a}}$ are linked by a mutation in direction $n$. 
 In the initial cluster $(x_{\overline{2n},2k},\:k\in\llbracket 1,n\rrbracket  )$, we have
\begin{equation}\label{rellambda}x_{\overline{2n},2n}x_{\overline{2n-2},2n-2}=x_{\overline{2n},2n-2}^2 +\lambda x_{\overline{2n},2n-2} +1.\end{equation}
The general relation (\ref{gexc2}) can be obtained directly in the following cluster (see Figure \ref{fig:mutclu}):
\begin{equation}\begin{array}{l}
\mu_{n-1}\mu_{n-2}\dots \mu_{k+1}(x_{\overline{2n},2k},\:k\in\llbracket 1,n\rrbracket)\\\qquad\qquad\quad = (x_{\overline{2n},2},x_{\overline{2n},4},\dots,x_{\overline{2n},2k}, x_{2k,2k+4},x_{2k,2k+6},\dots,x_{2k,2n},x_{\overline{2n},2n}),\end{array}\end{equation}
where performing the mutation $\mu_n$ maps $x_{\overline{2n},2n}$ to $x_{\overline{2k},2k}$, and $\theta_n^0$ gives (\ref{gexc2}).

Indeed, recall that   $\theta_n^0$ is unaffected by mutation, so that in order to understand $\mu_n$, it is enough to know how the matrix $B$ mutates, namely in the standard way (Definition 2). This determines the variables $x_{ab}$ appearing in the monomials $u_n^+$ and $u_n^-$ in the mutated cluster above, thus yielding (\ref{gexc2}).
\hfill $\Box$

\bigskip
\bigskip
\bigskip
\bigskip
\begin{figure}[!ht]
\centering
\includegraphics[bb=140 -20 440 280, width=4cm, height = 4cm, scale = 1]{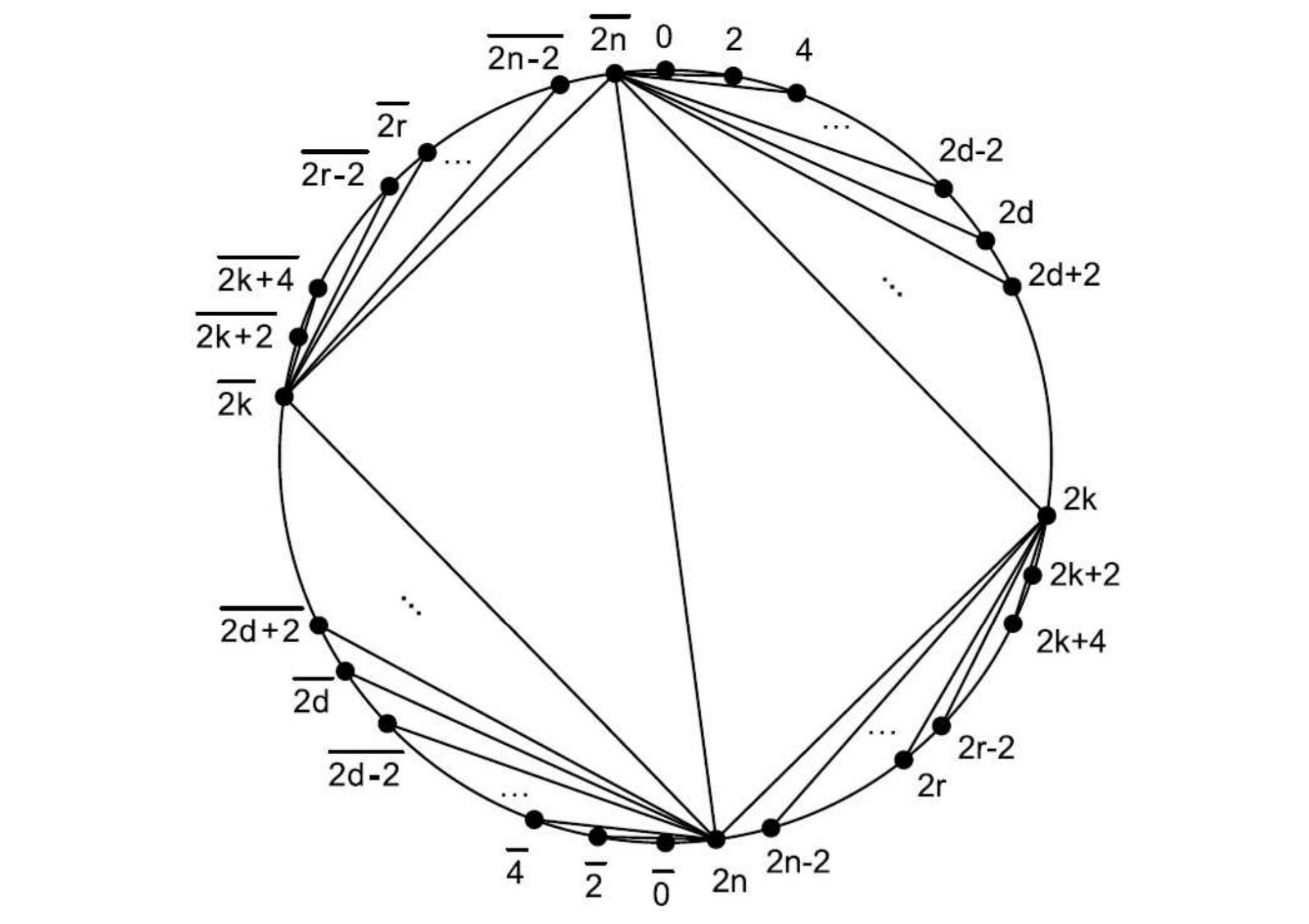}
\caption{The  mutated cluster $\mu_{n-1}\mu_{n-2}\dots \mu_{k+1}(x_{\overline{2n},2k},\:k\in\llbracket 1,n\rrbracket)$}
\label{fig:mutclu}\end{figure}

\end{proof}

\begin{remark}
Exchange relations do not cover every possibility for multiplication of cluster variables that are not in the same cluster. For example, we also have the following useful identity for multiplying a diagonal by a diameter: 
 if $a\neq \bar{b}$ and $c=\bar{d}$ (Figure \ref{fig:diagnex}), relations are of the form
\begin{equation}\label{rel2}
\begin{array}{ll} 
x_{\overline{2n},2k}x_{2d,\overline{2d}}=\lambda x_{\overline{2n},2d}x_{2d,2k} +   x_{\overline{2n},2d}x_{2k,\overline{2d}}+ x_{2d,2k}x_{\overline{2d},\overline{2n}}.
\end{array}
\end{equation}


\end{remark}

\bigskip
\bigskip
\bigskip

\begin{figure}[!ht]
\centering
\includegraphics[bb=120 00 420 300, width=4cm, height = 4cm, scale = 1]{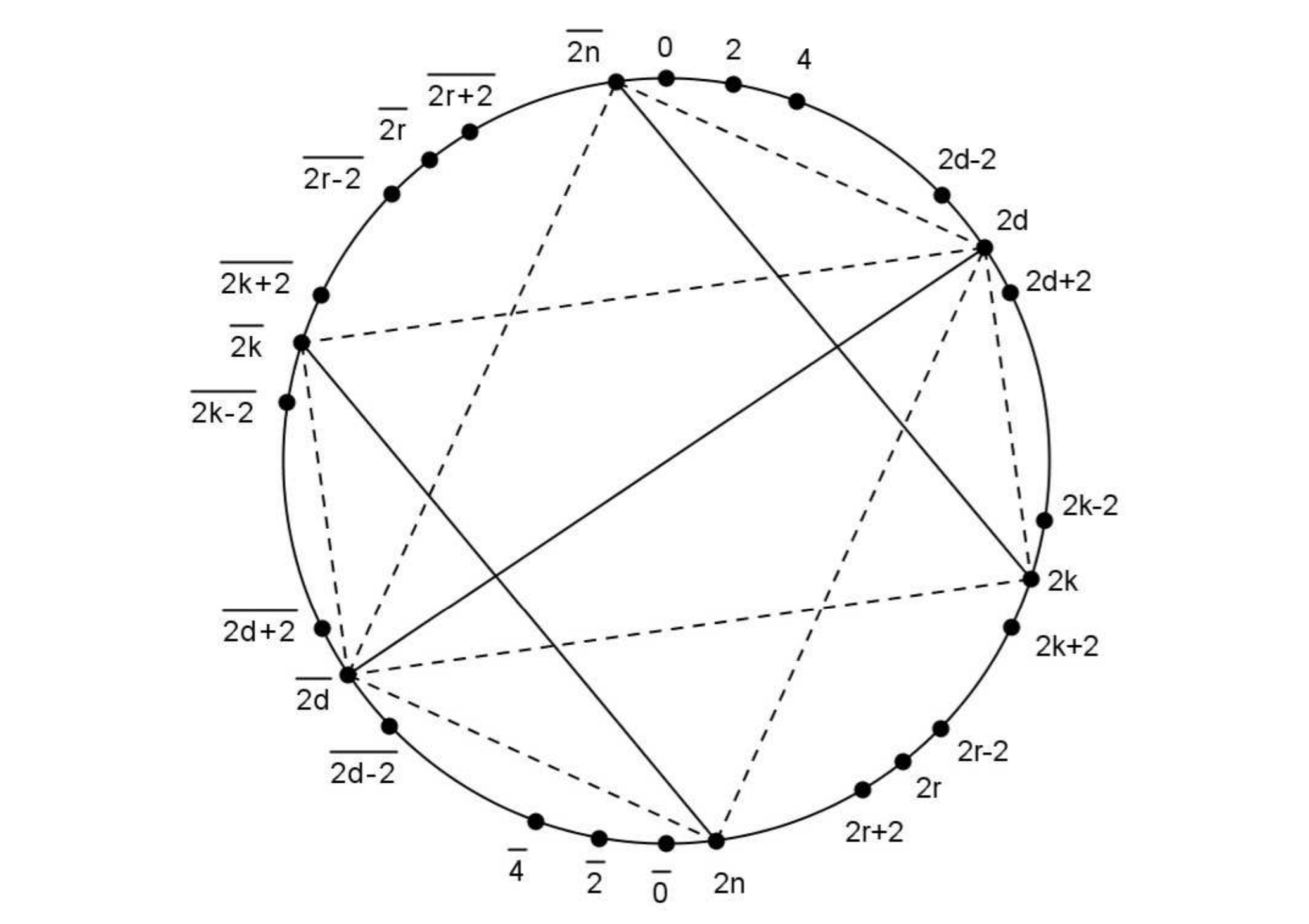}
\caption{Another relation for $\mathcal{A}_n$}
\label{fig:diagnex}\end{figure}


\subsubsection{Bases}\label{s232}

We call a pair of centrally symmetric diagonals of $\mathbf{P}_{2n+2}$ \emph{small} if 
the corresponding set $\mathcal{O}_{ab}$ contains only one element. The attached variables $x_{ab}$ are also called \emph{small}. Thus in type $C_3$, there are four small variables:
\begin{equation}
x_{04}, x_{26}, x_{4\bar{0}}, x_{6\bar{2}}.
\end{equation}

\begin{proposition}\label{genbasisA}
The set $\mathcal{S}$ of all monomials in the small variables  forms a $\mathbb{Z}$-basis of $\mathcal{A}_n$. Equivalently, $\mathcal A_n$ is the polynomial ring with coefficients in $\mathbb Z$ in the small variables.
\end{proposition}

\begin{proof}
We first prove that $\mathcal{S}$ spans $\mathcal{A}_n$ over $\mathbb{Z}$. Since $\mathcal{A}_n$ is generated by the elements $x_{ab}$, it is enough to show that each $x_{ab}$ is a polynomial in the small variables. We will argue by induction on $\mathrm{Card} \mathcal{O}_{ab}$.


Let $k\geq 2$, and suppose that variables $x_{ab}$ such that $\mathrm{Card}\:\mathcal{O}_{ab}\leq k-1$  can be written as  $\mathbb{Z}$-linear combinations of elements of $\mathcal{S}$.

Let $x_{ab}=x_{2d,2d+2k+2}$ be a cluster variable wih $\mathrm{Card}\:\mathcal{O}_{ab}=k\leq n$. Applying Proposition \ref{pCn} (1) in the  quadrilateron $\lbrack 2d,2d+2k-2,2d+2k,2d+2k+2\rbrack$ yields
\begin{equation}
\begin{array}{ll} x_{2d,2d+2k+2}  &= x_{2d+2k-2,2d+2k+2}\, x_{2d,2d+2k}   - x_{2d,2d+2k-2},\end{array}
\end{equation}
and by induction, the right-hand side is a $\mathbb Z$-linear combination of elements of $\mathcal S$. 
%
%
 Moreover,   $\lambda$ itself is a polynomial in the $x_{2r,2r+4}$: indeed, by \eqref{rel2}, we have 
\begin{equation}\label{lambda1} x_{\overline{2n},2n} x_{2n-2,\overline{0}} = \lambda + x_{\overline{0},\overline{2n}} + x_{\overline{2n},2n-2}.\end{equation}
Therefore, $\mathcal{S}$ spans $\mathcal{A}_n$ over $\mathbb{Z}$.

Let $\mathcal M_0$ be the set of cluster monomials of $\mathcal A_n$, and let $\mathcal M$ be the set of cluster monomials multiplied by powers of $\lambda$. 

Note that we can specialise $\lambda$ to 0 in $\mathcal A_n$, and this gives a standard cluster algebra $A_n$ of type $C_n$. Moreover, in this specialisation, the set $\mathcal M_0$ becomes the set $M$ of cluster monomials in $A_n$, which is free over $\mathbb Z$. This implies that $\mathcal M_0$ is free over $\mathbb Z\lbrack \lambda\rbrack$. Indeed, 
%
%
%
%
if $\mathcal{M}_0$ were not free, there would be a non-trivial $\mathbb{Z}\lbrack \lambda \rbrack$-linear dependence relation between elements of $\mathcal M_0$, of the form
\begin{equation}\label{m0}
\displaystyle \sum_{t=0}^N P_t(\lambda)\cdot m_t=0,\quad P_t\in\mathbb{Z}\lbrack\lambda\rbrack,\: m_t\in\mathcal{M}_0\quad(t\in\llbracket 0,N\rrbracket).
\end{equation}
Dividing if necessary by a suitable power of $\lambda$, we may assume that at least one $P_t(\lambda)$ is not divisible by $\lambda$, i.e. $P_t(0)=a_t\neq 0$.
The relation (\ref{m0}) above 
would then specialise, for $\lambda=0$, into a non-trivial  $\mathbb{Z}$-linear dependence relation between the cluster monomials of $M$. 
Thus $\mathcal{M}_0$ is free over $\mathbb{Z}\lbrack \lambda\rbrack$, and therefore $\mathcal{M}$ is free over $\mathbb{Z}$.

To show that $\mathcal S$ is free over $\mathbb Z$, let us now prove that $\mathcal{M}$ and $\mathcal{S}$ can be linked by an infinite unitriangular matrix $U$. 

For a monomial $m\in\mathcal{M}$ of the form $m=\lambda^e \cdot\prod x_{ab}^{m_{ab}}$, define its degree
\begin{equation}\deg(m):=(n+1)\cdot e+\sum m_{ab}\cdot\mathrm{Card}\:\mathcal{O}_{ab}.\end{equation}
Choose a total order on $\mathcal{M}$ such that for any $m,m'\in\mathcal{M}$,
\begin{equation} \deg(m)<\deg(m')\Rightarrow m< m' .\end{equation}

Let $\Phi:\mathcal{M}\rightarrow\mathcal{S}$ be the map that sends a monomial $m=\lambda^e \cdot\prod x_{ab}^{m_{ab}}\in\mathcal{M}$ to the monomial
\begin{equation}\Phi(m):= \displaystyle\left( \prod_{\mathrm{Card}\mathcal{O}_{ab}=1} x_{ab}\right)^e \cdot \prod\left( \prod_{ 2k\in\mathcal{O}_{ab}} x_{2k-2,2k+2} \right)^{m_{ab}}\in\mathcal{S}.\end{equation}
We show that $\Phi$ is a bijection by constructing an inverse  map $\Psi:\mathcal{S}\rightarrow\mathcal{M}$.

To a monomial $s=\displaystyle\prod_{k=0}^n x_{2k,2k+4}^{a_k}\in\mathcal{S}$, we attach the multiset $M(s)$ containing $a_k$ times the integer $2k+2$ for each $k=0,\dots,n$. A subset of $M(s)$ of the form
\begin{equation*}
\llbracket 2k,2\ell \rrbracket := \{2k,2k+2,\dots,2\ell-2,2\ell\}\quad(1\leq k\leq \ell\leq n+1)
\end{equation*}
is called a \emph{segment} of length $\ell-k+1$. Let $r$ be the number of distinct copies of $\llbracket 2,2n+2\rrbracket$ contained in $M(s)$, and let $M^{(1)}(s)$ be the multiset obtained from $M(s)$ by removing these $r$ maximal segments. Then it is an elementary combinatorial fact that $M^{(1)}(s)$ has a unique decomposition into a union of segments pairwise in \emph{generic position}. Here we say that two segments $\Sigma_1=\llbracket 2k_1,2\ell_1\rrbracket$ and $\Sigma_2=\llbracket 2k_2,2\ell_2\rrbracket$ are in generic position if the corresponding diagonals $(2k_1-2,2\ell_1+2)$ and $(2k_2-2,2\ell_2+2)$ do not intersect or are equal. Let $m_{ab}$ be 
 the number of copies of $\llbracket a+2,b-2\rrbracket$ in this decomposition. Then
\begin{equation*}
\Phi(s):=\lambda^r\displaystyle\prod x_{ab}^{m_{ab}}
\end{equation*}
is in $\mathcal M$ and $\Psi\circ\Phi(m)=m$, $\Phi\circ\Psi(s)=s$. We then order $\mathcal S$ by
\begin{equation*}
(s<s')\Leftrightarrow(\Phi(s)<\Psi(s')).
\end{equation*}




Let $U$ be the matrix $U=(u_{ms})_{m\in\mathcal{M},s\in\mathcal{S}}$ where the entries $u_{ms}\in\mathbb{Z}$ are defined by the infinite system of equations
\begin{equation}m= \displaystyle \sum_{s\in\mathcal{S}} u_{ms} s \quad(m\in\mathcal{M}).\end{equation}
The entries $u_{ms}$ are computed using the relations (\ref{gexc1}) to (\ref{rel2}) above. The rows and columns are ordered using the above total orders on $\mathcal{M}$ and $\mathcal{S}$.

By recursion on the degree, we are going to prove, with relations (\ref{gexc1})-(\ref{rel2}), that $U$ is lower unitriangular, that is, every monomial $m\in\mathcal{M}$ can be written as
\begin{equation}\label{monoMS}
m=\Phi(m) + \displaystyle\sum_{s<\Phi(m),\:s\in\mathcal{S}} u_{ms}\cdot s.
\end{equation}
First, if $m=x_{ab}$, we can deduce  from relation (\ref{gexc1}), that
\begin{equation}\label{monoxab}
x_{ab}=\Phi(x_{ab}) + \displaystyle \sum_{\deg(s)<\deg(x_{ab}),\:s\in\mathcal{S}} u_{x_{ab},s}\cdot s.
\end{equation}
Indeed, we can write $x_{ab}=x_{2k,2k+2d}$ for some $k,d\in\llbracket 0,n\rrbracket$. We have $x_{2k,2k+2}=1$ and $\Phi(x_{2k,2k+4})=x_{2k,2k+4}$. We also deduce from \eqref{gexc1} that
%
%
%
\begin{equation}x_{2k,2k+6}= x_{2k,2k+4}x_{2k+2,2k+6}-1=\Phi(x_{2k,2k+6})-1.\end{equation}
In general, suppose that the relation (\ref{monoxab}) holds, up to a certain degree $d-1<n$ of $x_{ab}.$ 
Then we use relation (\ref{gexc1}), up to shifting of the indices, to deduce that $x_{\overline{2n},2d}$ is equal to 
 $\Phi(x_{\overline{2n},2d})$ plus some terms of degree $<d$. This can be seen for each term displayed above.
 Therefore, we obtain (\ref{monoxab}). 

The reasoning is similar for $\lambda$. It suffices to take the variables in a special case of (\ref{gexc2}), and replace them with the expressions obtained from (\ref{monoxab}), to get an expression of the form
\begin{equation}\label{monolambda}
\lambda = (x_{\overline{2n},2} x_{04}x_{26}\dots x_{2n-4,\overline{2n}}x_{2n-2,\overline{0}}) + \displaystyle \sum_{\deg s \leq n} u_{\lambda,s} \cdot s.
\end{equation}
The right-hand side above is equal to $\Phi(\lambda)$, which is of degree $n+1$, plus some terms of degree $\leq n$, 
%
 hence (\ref{monolambda}) is true.
 The identities (\ref{monoxab}) and (\ref{monolambda}) yield the first lower triangular rows of $U$.


The relation \eqref{monoMS} now follows from \eqref{monoxab} and \eqref{monolambda} because of the compatibility of the orderings with multiplication. 
More precisely, note that we clearly have, for any two non-trivial monomials $m,m'\in\mathcal{M}$,
\begin{equation}\begin{array}{l}\Phi(m\cdot m')=\Phi(m)\Phi(m'),\quad \:\Phi(m)<\Phi(mm'),\quad\Phi(m')<\Phi(mm'),\\\mathrm{and}\: \:\:\deg(mm')=\deg(m)+\deg(m').\end{array}\end{equation}
Thus for any two cluster variables $x_{ab}$ and $x_{cd}$, we have
\begin{equation}\begin{array}{ll}
x_{ab}x_{cd}&=\Phi(x_{ab}x_{cd}) + \displaystyle\sum_{s<\Phi(x_{ab})} u_{x_{ab},s} s\Phi(x_{cd}) + \displaystyle\sum_{s'<\Phi(x_{cd})} u_{x_{cd},s'} \Phi(x_{ab}) s' \\\\&\qquad+ \displaystyle \sum_{s<\Phi(x_{ab}),s'<\Phi(x_{cd})} u_{x_{ab},s}u_{x_{cd},s'} s s',
\end{array}\end{equation}
where each term in the three sums is of degree $<\deg x_{ab}+\deg x_{cd}$. 

Likewise,  for each monomial $m_0\in\mathcal{M}_0$, the element $\lambda m_0$ is equal to $\Phi(\lambda m_0)$ plus some terms of degree $<n+1+\deg m_0$. This product compatibility immediately implies (\ref{monoMS}).

Finally, the unitriangularity of $U$ readily implies that, since $\mathcal M$ is free over $\mathbb Z$, then $\mathcal S$ is free over $\mathbb Z$. In conclusion, $\mathcal S$ is a $\mathbb Z$-basis of $\mathcal A_n$. \hfill$\Box$


\end{proof}

It follows from the proof of Proposition \ref{genbasisA} that the set $\mathcal M_0$ os cluster monomials is a $\mathbb Z\lbrack \lambda\rbrack$-basis of $\mathcal A_n$. We now use it to introduce another interesting $\mathbb Z$-basis of $\mathcal{A}_n$, which will be meaningful in representation theory.

For $k\in\mathbb{N}$, denote by $S_k(u)\in\mathbb{Z}\lbrack u\rbrack$ the $k$-th Tchebychev polynomial of the second kind
, given by
\begin{equation}S_k(u)^2=S_{k-1}(u)S_{k+1}(u)+1\end{equation}
with initial conditions $S_0(u)=1$ and $S_1(u)=u$.  
%
Recall from the proof of Proposition \ref{genbasisA}   that $\mathcal{M}_0$ is the set of cluster monomials of $\mathcal{A}_n$ that do not contain powers of $\lambda$. 
Then the set 
\begin{equation}\label{basisB}
\mathcal{B}:=\{ S_k(\lambda)\cdot m,\: k\in\mathbb{N}, m\in \mathcal{M}_0\}\end{equation}
is a $\mathbb Z$-basis of $\mathcal{A}_n$. 
This new basis will later correspond (Section \ref{s3}) to the basis of classes of simple modules in the Grothendieck ring of a category of representations of $\Uereslsl2$.

\section{Representations of quantum affine algebras}\label{s22}

Let $l$ be an integer, $l\geq 2$. Introduce the root of unity 
\begin{equation}\label{epsilon}\varepsilon :=\left\{ \begin{array}{ll}
\exp\left(\displaystyle\frac{i\pi}{l}\right) &\mbox{if } l \mbox{ is even}\\\\
\exp\left(\displaystyle\frac{2 i\pi}{l}\right) &\mbox{if } l \mbox{ is odd}
\end{array}        \right.\end{equation}
Thus $l$ is the order of $\varepsilon^2$, and we have $\varepsilon^{2l}=1$. Following \cite{FM}, let us also write
\begin{equation}\varepsilon^* := \varepsilon^{l^2}=1. \end{equation}

\subsection{Quantum affine algebras and their specialisations}\label{s221} 
Let $\mathfrak{g}$ be a finite-dimensional complex simple Lie algebra of simply-laced type, with 
Dynkin diagram $\delta$, vertex set $I=\llbracket 1,n\rrbracket$ and Cartan matrix $C=(a_{ij})_{i,j\in I}$. Denote by $\alpha_i$ the simple roots, by $\varpi_i$ the fundamental weights and by $P$ the weight lattice \cite{B}. 



Let $q$ be an indeterminate; then $\mathbb{C}(q)$ is the field of rational functions of $q$ with complex coefficients, and $\mathbb{C}\lbrack q,q^{-1}\rbrack$ is the ring of complex Laurent polynomials in $q$.

Let $U_q(\widehat{\mathfrak{g}})$ be the quantum affine algebra associated with ${\mathfrak g}$ \cite{FR2}. This is a Hopf algebra over $\mathbb C (q)$. Denote by $\Uqlg$ the quantum loop algebra, which is isomorphic to a quotient of  $U_q(\widehat{\mathfrak{g}})$ where the central charge is mapped to 1. Therefore, $\Uqlg$ inherits a Hopf algebra structure. For more information on $L\mathfrak g$, $\widehat{\mathfrak g}$ and their quantum enveloping algebras, we refer the reader to \cite{CP}, \cite{CH} and \cite{Lec}.

We will be interested in finite-dimensional representations of  $U_q(\widehat{\mathfrak{g}})$, on which the central charge acts trivially. It is therefore sufficient to consider finite-dimensional representations of $\Uqlg$, and we will focus on these onwards.

Let $\Uqreslg$ be the restricted integral form corresponding to $\Uqlg$ \cite{CP2}. This is a Hopf algebra over $\mathbb C \lbrack q,q^{-1}\rbrack$.

Let us now specialise $\Uqreslg$ at the root of unity $\varepsilon$, by setting
\begin{equation}\begin{array}{l}
\Uereslg:= \Uqreslg\otimes_{\mathbb{C}\lbrack q,q^{-1}\rbrack} \mathbb{C}\\
\end{array}\end{equation}
via the algebra homomorphism 
\begin{equation}\begin{array}{ccc}
\mathbb{C}\lbrack q,q^{-1}\rbrack&\longrightarrow&\mathbb{C}\\
q&\longmapsto&\varepsilon
\end{array}\end{equation}
For an element $x$ of $\Uqreslg$
, we denote the corresponding element of $\Uereslg$
 also by $x$.

\subsection{Representations of \texorpdfstring{$\Uqlg$}{Uq(Lg)}}\label{s222} 

\paragraph{The category $\mathcal{C}_q$.} Let $\mathcal{C}_q$  be the category of finite-dimensional type 1 $\Uqlg$-modules (see \cite[Section 11.2]{CP}). It is known that $\mathcal{C}_q$ is a monoidal, abelian, non semisimple category.

An object $V$ in $\mathcal{C}_q$ has a $q$-character $\chi_q(V)$ \cite{FR2}, which is a Laurent polynomial with positive integer coefficients in variables $Y_{i,a}$, $i\in I$, $a\in\mathbb C(q)$. Any irreducible object or $\mathcal{C}_q$ is determined, up to isomorphism, by its $q$-character. Such irreducible representations are parametrised \cite{FR2} by the highest dominant monomial of their $q$-characters, which is a \emph{dominant monomial}, i.e. it contains only positive exponents.

Let $\mathcal{M}_q$ be the set of Laurent monomials in the $Y_{i,a}$, and let $\mathcal{M}_q^+$ be the subset of dominant monomials in $\mathcal{M}_q$. If $S$ is a simple object of $\mathcal{C}_q$ such that the highest monomial of $\chi_q(S)$ is $m\in\mathcal{M}_q^+$, then $S$ will be denoted by $L(m)$ \cite{Lec}.   
For $i\in I$ and $a\in\mathbb{C}(q)$, the simple modules $L(Y_{i,a})$ are called fundamental modules. A \emph{standard module} is a tensor product of fundamental modules.

Let $K_0(\mathcal{C}_q)$ be the Grothendieck ring of $\mathcal{C}_q$. It is known ({\cite[Corollary 2]{FR2}}) that 
\begin{equation}K_0(\mathcal{C}_q) \cong \mathbb{Z}\lbrack \lbrack L(Y_{i,a})\rbrack, i\in I, a\in\mathbb{C}(q)\rbrack.\end{equation}

For $i\in I,\:k\in\mathbb{N}^*,\: a\in\mathbb{C}(q)$, the simple object
\begin{equation}W_{k,a}^{(i)} = L(Y_{i,a} Y_{i,aq^2}\dots Y_{i,aq^{2(k-1)}})\end{equation}
is called a \emph{Kirillov-Reshetikhin module}. 
In particular, we have $W_{1,a}^{(i)}=L(Y_{i,a})$ and by convention, $W_{0,a}^{(i)}=\mathbf{1}$ for any $a,i$.

The classes $\lbrack W_{k,a}^{(i)}\rbrack$ in $K_0(\mathcal C_q)$ satisfy a system of equations called $T$-system: 
\begin{equation}\lbrack W_{k,a}^{(i)}\rbrack\lbrack W_{k,aq^2}^{(i)}\rbrack = \lbrack W_{k+1,a}^{(i)}\rbrack\lbrack W_{k-1,aq^2}^{(i)}\rbrack+\displaystyle \prod_{j\sim i} \lbrack W_{k,aq}^{(j)}\rbrack\quad (i\in I,k\in\mathbb{N}^*,a\in\mathbb{C}(q)),\end{equation}
where $j\sim i$ means that $j$ is a neighbour of $i$ in the Dynkin diagram $\delta$.

\begin{example}
If ${\mathfrak{g}}={\mathfrak{sl}_2}$, the $T$-system reads
\begin{equation}\lbrack W_{k,a}^{(1)}\rbrack\lbrack W_{k,aq^2}^{(1)}\rbrack = \lbrack W_{k+1,a}^{(1)}\rbrack\lbrack W_{k-1,aq^2}^{(1)}\rbrack+1\quad (k\in\mathbb{N}^*).\end{equation}
\end{example}


\paragraph{The category $\CZ$.} Let us now define a subcategory of $\mathcal{C}_q$, following \cite{HL}.

Since the Dynkin diagram $\delta$ is a bipartite graph, there is a partition of the vertices $I=I_0 \sqcup I_1$, where each edge connects a vertex of $I_0$ with a vertex of $I_1$. For $i\in I$, set
\begin{equation}\xi_i:=\left\{ \begin{array}{ll}
0&\mbox{if } i\in I_0\\
1&\mbox{if }i\in I_1
\end{array}\right.\end{equation}
The map $i\mapsto \xi_i$ is determined by the choice of $\xi_{i_0}\in\{0,1\}$ for a single vertex $i_0$. Therefore, there are only two possible collections of $\xi_i$.

Let $\CZ$ be the full subcategory of $\mathcal{C}_q$ whose objects $M$ have all their composition factors $L(m)$ such that  $m$ contains only variables of the form $Y_{i,q^{2k+\xi_i}}$ $(k\in\mathbb Z,\:i\in I)$. This is a tensor subcategory of $\mathcal{C}_q$.

The ring $R_\mathbb{Z} := K_0(\mathcal{C}_{q^{\mathbb{Z}}})$ is the subring of $K_0(\mathcal{C}_q)$ generated by the classes of the form $\lbrack L(Y_{i,q^{2k+\xi_i}})\rbrack\:(i\in I,k\in\mathbb{Z}).$


\subsection{Representations of \texorpdfstring{$\Uereslg$}{Ueres(Lg)}}\label{s223}

Let $\Cres$ be the category of finite-dimensional type 1 $\Uereslg$-modules. Let $K_0(\Cres)$ be its Grothendieck ring. 

An object $V$ in $\Cres$ has an $\varepsilon$-character $\chi_\varepsilon(V)$ \cite{FM}, which is a Laurent polynomial with positive integer coefficients in variables $Y_{i,a}$, $i\in I$, where $a\in\mathbb C^*$. 

The parametrisation of the simple objects by their highest  monomials also holds on $\mathcal C_\varepsilon$, with $q$ replaced by $\varepsilon$  (see \cite{CP2,FM}). 
Let $\mathcal{M}_\varepsilon$ be the set obtained from $\mathcal{M}_q$ by replacing $q$ by $\varepsilon$, and let $\mathcal M_\varepsilon^+$ be the subset of dominant monomials in $\mathcal M_\varepsilon$. 
The simple module whose highest weight monomial is $m\in\mathcal{M}_\varepsilon^+$ will be denoted by $L(m)$.

In particular, the \emph{fundamental modules} of $\Uereslg$ are the simple objects \\$L(Y_{i,a})$, where $i\in I,\:a\in\mathbb{C}^*$, and the \emph{standard modules} are the tensor products of fundamental modules. The simple module $L(m)$ is called 
\emph{prime} if it cannot be written as a tensor product of non-trivial modules. 

The ring $K_0(\Cres)$ is the ring of polynomials with integer coefficients in variables $\lbrack L(Y_{i,a})\rbrack$, $i\in I,a\in\mathbb{C}^*$ (see \cite[Section 3.1]{FM}).

For a simple object $V$ of $\mathcal{C}_q$, with highest weight vector $v$, it is known \cite[Proposition 2.5]{FM} that the $\Uqreslg$-module $V^{\mathrm{res}}:= \Uqreslg\cdot v$   is a free $\mathbb C\lbrack q,q^{-1}\rbrack$-module.  
Put $V_\varepsilon^{\mathrm{res}}=V^{\mathrm{res}}\otimes_{\mathbb{C}\lbrack q,q^{-1}\rbrack} \mathbb C$, where as above $q$ acts on $\mathbb C$ by multiplication by $\varepsilon$. This is a $\Uereslg$-module called the \emph{specialisation of $V$ at $q=\varepsilon$}. 

For $i\in I$ and $a\in\mathbb{C}^*$, introduce the following notation:
\begin{equation}\mathbf{Y}_{i,a}:= \displaystyle \prod_{j=0}^{l-1} Y_{i,a\varepsilon^{2j+\xi_i}}.\end{equation}
Note that since $\varepsilon^{2l}=1$, we have $\mathbf{Y}_{i,\varepsilon^{2r}}=\mathbf{Y}_{i,1}$ for any $r\in\mathbb{Z}$. 
A monomial in the variables $Y_{i,a}$ is called \emph{$l$-acyclic} if it is not divisible by $\mathbf{Y}_{j,b}$ for any $j\in I,\:b\in\mathbb C^*$.

Let $\CZres$ be the full subcategory of $\Cres$ whose objects $M$ have all their composition factors $L(m)$ such that $m\in\mathcal{M}_\varepsilon^+$ contains only variables of the form $Y_{i,\varepsilon^{2k+\xi_i}}$, where $i\in I$ and $k\in\mathbb{Z}$. 
 For example, the modules $z_i:= L(\mathbf{Y}_{i,1})$ are objects of $\CZres$.

Let $R_{\varepsilon^{\mathbb{Z}}}=K_0(\CZres)$ be the Grothendieck ring of $\CZres$. This is the subring of $K_0(\Cres)$ generated by the classes $\lbrack L(Y_{i,\varepsilon^{2k+\xi_i}})\rbrack, i\in I,\:k\in\mathbb{Z}.$


\subsection{Representations of \texorpdfstring{$\Ueereslg$}{Ueres(Lg)}}\label{s225}

Consider the category $\mbox{Rep }\Ueereslg$ of finite-dimensional type 1 representations of $\Ueereslg$. Since we are in the case where $\varepsilon^{*}=\varepsilon^{l^2}=1$, this category is equivalent to the category $\Rep L{\mathfrak{g}}$ of finite-dimensional $ L{\mathfrak{g}}$-modules.

For $a\in\mathbb{C}^*$, consider the evaluation morphism $\phi_a: L{\mathfrak{g}}\cong \mathfrak{g}\lbrack t,t^{-1}\rbrack \rightarrow \mathfrak{g}$ that maps a Laurent polynomial $P(t)$ to its evaluation $P(a)$ at $a$. 
 For an irreducible representation $V_\lambda$ of $\mathfrak{g}$ that has highest weight $\lambda$, the pullback $V_\lambda(a):=\phi_a^*(V_\lambda)$ is an irreducible $L {\mathfrak{g}}$-module. It is known \cite{CP1} that any simple object $S\in\Rep L{\mathfrak{g}}$ is a tensor product of evaluation modules $V_{\lambda_1}(a_1)\otimes \dots \otimes V_{\lambda_n}(a_n)$, such that $a_i\neq a_j$ for every $i\neq j$.

Let $V$ be a representation of $\mathfrak{g}$, with weight decomposition $V=\bigoplus_\mu V_\mu$. Recall that the character $\chi(V)$ of $V$ is a polynomial in variables $y_i^{\pm 1},i\in I$ defined by  \begin{equation}\chi(V)=\sum_\mu \dim V_\mu \cdot y^\mu,\end{equation} where for a weight $\mu=\sum_{i\in I} \mu_i\varpi_i$, we set $y^\mu = \prod_{i\in I}y_i^{\mu_i}$.

%


\subsection{From \texorpdfstring{$q$}{q}-characters to \texorpdfstring{$\varepsilon$}{e}-characters}\label{s226}

Frenkel and Mukhin (\cite[Proposition 2.5]{FM}) prove that there is a surjective ring morphism $K_0(\mathcal{C}_q)\rightarrow K_0(\Cres)$ that maps the isomorphism class $\lbrack V\rbrack$ of a simple object $V$ of $\mathcal{C}_q$ to the class $\lbrack V_\varepsilon^{\mathrm{res}}\rbrack$.

Since the map $\chi_q:K_0(\mathcal{C}_q)\rightarrow \mathbb{Z}\lbrack Y_{i,a}^{\pm 1},\: i\in I,\:a\in\mathbb{C}(q)\rbrack$ is an injective ring morphism (see \cite[Theorem 3]{FR2}), Frenkel and Mukhin (\cite[Theorem 3.2]{FM}) prove that the $\varepsilon$-character map $\chi_\varepsilon:K_0(\Cres)\rightarrow \mathbb{Z}\lbrack Y_{i,a}^{\pm 1},\:i\in I,\:a\in\mathbb{C}^*\rbrack$ is also an injective ring morphism.
Moreover, Theorem 3.2 in \cite{FM} also states that for a simple module $V\in\mathcal{C}_q$, the $\varepsilon$-character $\chi_\varepsilon(V_\varepsilon^{\mathrm{res}})$ is obtained by substituting $q\mapsto \varepsilon$ in $\chi_q(V)$. The $\varepsilon$-characters $\chi_\varepsilon$ thus satisfy combinatorial properties similar to $q$-characters (\cite[Section 3.2]{FM}).

\subsection{The Frobenius pullback}\label{s227} Following Lusztig \cite{Lu}, Frenkel and Mukhin \cite{FM} describe a \emph{quantum Frobenius map} $\mathrm{Fr}:\Uereslg\rightarrow\Ueereslg$ that gives rise to the \emph{Frobenius pullback} \begin{equation}\mbox{Fr}^* : K_0(\Rep\Ueereslg)\longrightarrow K_0(\Cres).\end{equation}

\begin{proposition}[{\cite[Lemma 4.7]{FM}}]\label{frob}
The Frobenius pullback 
\begin{equation}\mathrm{Fr}^* :\begin{array}{ccc} K_0(\Rep\Ueereslg) &\longrightarrow&  K_0(\Cres)
\end{array}
\end{equation}
 is the injective ring homomorphism such that $\mathrm{Fr}^*(\lbrack L(Y_{i,a}) \rbrack)= \lbrack L(\mathbf{Y}_{i,a})\rbrack$.
\end{proposition}


\subsection{Decomposition theorem}\label{s228} 
Let $m\in\mathcal M_\varepsilon^+$. There is a unique factorisation $m=m^0 m^1$ where $m^1$ is a monomial in the variables $\mathbf{Y}_{i,a}$, and $m^0$ is $l$-acyclic.  
The monomial $m^0$ is called the $l$-acyclic part of $m$. The following theorem was proved by Chari and Pressley for roots of unity of odd order \cite{CP2} and generalised by Frenkel and Mukhin \cite{FM} to roots of unity of arbitrary order.

\begin{theorem}[{\cite[Theorem 5.4]{FM}}] \label{decthm}
Let $L(m)$ be a simple object of $\Cres$. Then \begin{equation}L(m)\cong L(m^0)\otimes L(m^1).\end{equation}

\end{theorem}

Note that by Proposition \ref{frob}, $L(m^1)$ is the Frobenius pullback of an irreducible $\Ueereslg$-module.



\subsection{Characters of \texorpdfstring{$\Ueereslg$}{Ue*reslg}-modules} 
Since $\varepsilon^*=1$, the category $\Rep\Ueereslg$ is equivalent to $\Rep L{\mathfrak{g}}$, and in order to compute $\varepsilon^*$-characters we just need to know $\chi_1(V_\lambda(a))$, which is obtained from $\chi(V_\lambda)$ by replacing each $y_i^{\pm 1}$ by $Y_{i,a}^{\pm 1}$.

For an irreducible $\Ueereslg$-module $L(m)$, the pullback $\mathrm{Fr}^*(L(m))$ is the module $L(M)$, where $M$ is the monomial obtained from $m$ by replacing each $Y_{i,a^l}^{\pm 1}$ with $\mathbf{Y}_{i,a}^{\pm 1}$. 
The $\varepsilon$-character $\chi_\varepsilon(\mathrm{Fr}^*(L(m)))$ of $L(m)\in \Rep\Ueereslg$ is obtained from $\chi_{\varepsilon^*}(L(m))$ by replacing each $Y_{i,a^l}^{\pm 1}$ by $\mathbf{Y}_{i,a\varepsilon^{\xi_i}}^{\pm 1}$.

Therefore, the computation of $\varepsilon$-characters of simple objects of  $\Cres$ is reduced, by Theorem \ref{decthm}, to understanding the $\varepsilon$-characters of all representations $L(m^0)$ where $m^0$ is $l$-acyclic.

\subsection{The case  \texorpdfstring{$\mathfrak g = \mathfrak{sl}_2$    }{g=sl2}}\label{s310}

\subsubsection{Prime simple modules}

Let us introduce the following notation:
\begin{equation}W_\varepsilon(k,a,i):=\left( W_{k,a}^{(i)}\right)^{\mathrm{res}}_\varepsilon\quad(i\in I,\:k\in\mathbb N,\: a\in\mathbb C^*).\end{equation}
Here $I=\{1\}$, so we drop the index $i$ in the above notation. The specialisation of a Kirillov-Reshetikhin module of $\CZ$ to an object of $\CZres$ is not always simple. 

To each module $W_\varepsilon(k,\varepsilon^{2d})$ with $k<l$ and $d\in\llbracket 1,l\rrbracket$, attach the diagonal $\lbrack 2d-2,2d+2k\rbrack$ of the $2l$-gon $\mathbf{P}_{2l}$ defined in Section \ref{s231}. 
The following result is equivalent to a special case of a theorem from Chari and Pressley \cite{CP2}.
\begin{theorem}[{\cite[Theorem 9.6]{CP2}}] \label{KRA}
For $\mathfrak g = \mathfrak{sl}_2$, the simple objects of $\CZres$ are exactly the tensor products of the form
\begin{equation}
\displaystyle \bigotimes_{t=1}^r L(Y_{1,\varepsilon^{2d_t}} \dots Y_{1,\varepsilon^{2(d_t+k_t-1)}})^{\otimes a_t}\otimes L(\mathbf{Y}_{1,1}^{ a}) = \displaystyle \bigotimes_{t=1}^r W_\varepsilon (k_t,\varepsilon^{2d_t})^{\otimes a_t}\otimes L(\mathbf{Y}_{1,1}^{ a})
\end{equation}
where $r\in\mathbb N^*,\: k_1,\dots,k_r\in\llbracket 0,l-1\rrbracket,\:d_1,\dots,d_r\in\mathbb Z,\: a_1,\dots,a_r,a\in\mathbb N$, under the condition that for every $t\neq s\in\llbracket 1,r\rrbracket$, the diagonals $\lbrack 2d_t-2,2d_t+2k_t\rbrack$ and $\lbrack 2d_s-2,2d_s+2k_s\rbrack$ do not intersect inside $\mathbf{P}_{2l}$. 

\end{theorem}

It follows from Theorem \ref{KRA} that the \emph{prime} simple modules of $\Uereslsl2$ are the modules $W_\varepsilon(k,\varepsilon^{2d})$ $(k<l,\:r\in\llbracket 1,l\rrbracket)$ and the Frobenius pullbacks $L(\mathbf{Y}_{1,1}^a)$ $(a\in\mathbb N^*)$. 
%
%
From now on, we drop the index $i=1$ in the variables $Y_{1,\varepsilon^n}$ and introduce the notation
\begin{equation}Y_{n}:= Y_{1,\varepsilon^{n}}\quad\mbox{and}\quad \mathbf{Y}_1=Y_0Y_2\dots Y_{2l-2}\end{equation}
for any integer $n$. We then have $Y_{2l+n}=Y_n$ for every $n$.  

With this new notation, we have
\begin{equation}\label{eq:KR}\begin{array}{l}W_\varepsilon(k,\varepsilon^{2r})= L(Y_{2r}Y_{2r+2}Y_{2r+4}\dots Y_{2(r+k-1)}) \quad(k,r\in\llbracket 0,l-1\rrbracket),\\ z:=z_1= L(\mathbf{Y}_{1}).\end{array}\end{equation}


\subsubsection{\texorpdfstring{$\varepsilon$}{e}-Characters}
In type $A_1$, we have an explicit expression for $q$-characters:
\begin{equation}\label{expA1}\begin{array}{ll}
\chi_q(W_{k,a}^{(1)}) &= Y_{1,a}Y_{1,aq^2}\dots Y_{1,aq^{2(k-2)}}Y_{1,aq^{2(k-1)}} \\&\quad+Y_{1,a}Y_{1,aq^2}\dots Y_{1,aq^{2(k-2)}}Y_{1,aq^{2k}}^{-1}\\&\quad+Y_{1,a}Y_{1,aq^2}\dots Y_{1,aq^{2(k-3)}}Y_{1,aq^{2(k-1)}}^{-1}Y_{1,aq^{2k}}^{-1}\\&\quad+\dots + Y_{1,aq^2}^{-1}Y_{1,aq^4}^{-1}\dots Y_{1,aq^{2(k-1)}}^{-1}Y_{1,aq^{2k}}^{-1} .
\end{array}\end{equation}
Each module $W_{k,q^{2d}}^{(1)}$ specialises to $W_\varepsilon(k,\varepsilon^{2d})$, which is irreducible if $k<l$. We can then directly translate the formula above into the $Y_n$ notation for $k<l$.
 Moreover, since $z_1=L(\mathbf{Y}_{1,1})$ is the pullback of $L(y_1)$, and $\chi(L(y_1))=y_1+y_1^{-1}$, we have 
\begin{equation}\label{foz1}\chi_\varepsilon( z_1) = Y_0Y_2\dots Y_{2(l-1)}+Y_0^{-1}Y_2^{-1}\dots Y_{2(l-1)}^{-1}.\end{equation}
This implies that the specialised Kirillov-Reshetikhin modules $W_\varepsilon(k,\varepsilon^{2d})$, for $k<l-1$, satisfy the $T$-system, and the $\varepsilon$-characters behave the same way as their $q$-character counterparts. Namely, for $k\leq l-2$, we have
\begin{equation}\label{tsyse}\begin{array}{l}\chi_\varepsilon(L(Y_0Y_2\dots Y_{2k}))\chi_\varepsilon(L(Y_2Y_4\dots Y_{2k+2} ))\\\qquad =\chi_\varepsilon(L(Y_2\dots Y_{2k}))\chi_\varepsilon(L(Y_0\dots Y_{2k+2}))+1.\end{array}\end{equation}
The difference of behaviour between $q$-characters and $\varepsilon$-characters resides in the $(l-1)$-th equation of the $T$-system, 
which will later be seen as a generalised exchange relation:

\begin{lemma}\label{carac}
We have 
\begin{equation}\label{caracl}\begin{array}{ll}\chi_\varepsilon(L(Y_0Y_2Y_4\dots Y_{2(l-2)}))\chi_\varepsilon(L(Y_2Y_4\dots Y_{2(l-1)}))\\\quad= 
\chi_\varepsilon(L(Y_2Y_4\dots Y_{2(l-2)})) \cdot \chi_\varepsilon( z_1) +1 + \chi_\varepsilon(L(Y_2Y_4\dots Y_{2(l-2)}))^2.\end{array}\end{equation}
\end{lemma}

\begin{proof}
This follows immediately from an explicit computation using formulas \eqref{expA1} and \eqref{foz1}. \hfill$\Box$

\end{proof}

\section{The cluster structure on \texorpdfstring{$K_0(\CZres)$}{Ko(CZres)} in type \texorpdfstring{$A_1$}{A1}}\label{s3}

We can now state the main theorem of this paper. Recall the generalised cluster algebra $\mathcal{A}_n$ defined in Section \ref{s213}, Definition \ref{dAn}. Let $R=R_{\varepsilon^{\mathbb Z}}$ be the Grothendieck ring of $\CZres$ for $\mathfrak g = \mathfrak {sl}_2$ and $\varepsilon$ as in \eqref{epsilon}.

\begin{theorem}\label{conjA1}
There exists a ring  isomorphism $\varphi:\mathcal{A}_{l-1}\rightarrow R$, such that
\begin{equation}\varphi (x_{2r,2d})=\lbrack L(Y_{2r+2}\dots Y_{2d-2})\rbrack\:(r,d\in\llbracket 0,l-1\rrbracket, |r-d|<l),\quad \varphi(\lambda)= \lbrack z_1\rbrack.\end{equation}
The $\mathbb{Z}$-basis $\mathcal{S}$ of $\mathcal{A}_{l-1}$ is mapped by $\varphi$ to the basis of  classes of standard modules in $R$. The $\mathbb{Z}$-basis $\mathcal{B}$ of $\mathcal{A}_{l-1}$ consisting in generalised cluster monomials (see (\ref{basisB})) is mapped to the basis $B$ of classes of simple modules in $R$.

\end{theorem}

\begin{proof}

We established an isomorphism $\mathcal{A}_{l-1}\cong \mathbb{Z}\lbrack x_{2r-2,2r+2},\:r\in\llbracket 0,l-1\rrbracket\rbrack$ in Section \ref{s213} (Proposition \ref{genbasisA}). We also know from Section \ref{s223} that there is an isomorphism $R\cong \mathbb{Z}\lbrack \lbrack L(Y_{2k})\rbrack,\:k\in\llbracket 0,l-1\rrbracket \rbrack$. Therefore, we 
 may fix a ring isomorphism $\varphi$, which sends each variable $x_{2r-2,2r+2}$ to the class $\lbrack L(Y_{2r})\rbrack$ in the Grothendieck ring $R$. Clearly, $\varphi$ maps the basis $\mathcal S$ to the basis of classes of standard modules.

It is easy to deduce from Proposition \ref{pCn} that the cluster variables in $\mathcal{A}_n$ are built from the $x_{2r-2,2r+2}$ using relation (\ref{gexc1}), with $k=r-d$. On the other hand, \eqref{tsyse}  
implies that the $\varepsilon$-characters of 
the simple Kirillov-Reshetikhin modules in $R$ are built from the $\varepsilon$-characters of fundamental modules $\chi_\varepsilon(L(Y_{2r}))$ using the same relations. Thus $\varphi(x_{2r,2d})=\lbrack L(Y_{2r+2}\dots Y_{2d-2})\rbrack$. Moreover, comparing \eqref{rellambda} and \eqref{caracl}, we also get $\varphi(\lambda)=\lbrack z_1\rbrack$. 

Let us now move on to the correspondence between the bases $\mathcal B$ and $B$. We know from Theorem \ref{decthm} that the class of every simple module can be written as $F\cdot M$, where $M$ is 
$l$-acyclic 
 and $F$ is the Frobenius pullback of the class of an irreducible $\mathfrak{sl}_2$-module. We know that $\lbrack z_1\rbrack$ corresponds to $\mathrm{Fr}^*(\lbrack V(\varpi)\rbrack)$, the Frobenius pullback of the 2-dimensional representation of $\mathfrak{sl}_2$. 
It then follows from the classical theory of characters for $\mathfrak{sl}_2$ that $\mathrm{Fr}^*(\lbrack V(k\varpi)\rbrack)$ is the Tchebychev polynomial of the second kind $S_k(\lbrack z_1\rbrack)$. Therefore, the basis of classes of simple modules in $R$ consists of elements of the form $S_k(\lbrack z_1\rbrack)\cdot M$, where $M$ is the class of a tensor product of simple Kirillov-Reshetikhin modules that satisfy the condition from Theorem \ref{KRA}. This geometrical condition on diagonals of $\mathbf{P}_{2l}$ corresponds exactly to cluster variable compatibility in $\mathcal{A}_{l-1}$: indeed, recall that two cluster variables of  $\mathcal{A}_{l-1}$ are in the same cluster if and only if their attached diagonals 
do not cross inside $\mathbf{P}_{2l}$. 
Therefore, Theorem \ref{decthm} 
 allows us to conclude that the image of the basis $\mathcal{B}$  under the isomorphism $\varphi$, is the basis $B$ of classes of simple modules in $R$. \hfill $\Box$

\end{proof}


\section{Type \texorpdfstring{$A_2$}{A2}}\label{s4}


\subsection{The case \texorpdfstring{$A_2$, $l=2$.}{A2,l=2.}}\label{s42}
We start by studying the Grothendieck ring of $\CZres$ for $\mathfrak g = \mathfrak{sl}_3$ and $l=2$, in terms of $\varepsilon$-characters. For any integer $n\in\mathbb Z$ and any vertex $i=1,2$ of the Dynkin diagram, we set
\begin{equation*}
Y_{i,n}:=Y_{i,\varepsilon^n},
\end{equation*}
where  we consider the second index modulo $4$. We also write
\begin{equation*}
\mathbf{Y}_1=Y_{1,0}Y_{1,2}\quad\mbox{and}\quad\mathbf{Y}_2=Y_{2,1}Y_{2,3}.
\end{equation*}
In this case, the simple modules of $\CZres$ are of the form $L(m)$, where $$m=Y_{1,0}^{a_{10}}Y_{1,2}^{a_{12}}Y_{2,1}^{a_{21}}Y_{2,3}^{a_{23}}\quad(a_{10},\dots,a_{23}\in\mathbb N).$$ 

\begin{lemma}\label{caracA2}
The Laurent polynomial $\chi_\varepsilon(L(Y_{1,0}))^k\chi_\varepsilon(L(Y_{1,0}Y_{2,1}))^\ell$ contains a unique dominant monomial. It follows that 
\begin{equation}
\chi_\varepsilon(L(Y_{1,0}))^k\chi_\varepsilon(L(Y_{1,0}Y_{2,1}))^\ell = \chi_\varepsilon(L(Y_{1,0}^{k+\ell}Y_{2,1}^\ell)).
\end{equation}
\end{lemma}

\begin{proof}
It is easy to compute the following $\varepsilon$-characters:
\begin{equation}\label{A22}\begin{array}{lll}
\chi_\varepsilon(L(Y_{1,0}))&=&  Y_{1,0} + Y_{2,1}Y_{1,2}^{-1} + Y_{2,3}^{-1}\\
\chi_\varepsilon(L(Y_{1,0}Y_{2,1}))& =& Y_{1,0}Y_{2,1} + Y_{1,0}Y_{1,2}Y_{2,3}^{-1} + Y_{2,1}^2 Y_{1,2}^{-1} + 2 Y_{2,1}Y_{2,3}^{-1}\\&&\quad + Y_{2,1}Y_{1,0}^{-1}Y_{1,2}^{-1} + Y_{1,2}Y_{2,3}^{-2} + Y_{1,0}^{-1} Y_{2,3}^{-1}.
\end{array}
\end{equation}
One can directly check that the only way to obtain a dominant monomial by multiplying terms of the sums above, is to involve only $Y_{1,0}$ and $Y_{1,0}Y_{2,1}$. Thus the only dominant monomial of $\chi_\varepsilon(L(Y_{1,0}))^k\chi_\varepsilon(L(Y_{1,0}Y_{2,1}))^\ell$ is $Y_{1,0}^{k+\ell}Y_{2,1}^\ell$. Therefore, the $\varepsilon$-character  $\chi_\varepsilon(L(Y_{1,0}))^k\chi_\varepsilon(L(Y_{1,0}Y_{2,1}))^\ell$coincides with $\chi_\varepsilon(L(Y_{1,0}^{k+\ell}Y_{2,1}^\ell))$.  
\hfill $\Box$
\end{proof}

Theorem \ref{decthm} allows us to deduce the following property.

\begin{theorem}\label{decG}
Any simple finite-dimensional $U_\varepsilon^{\mathrm{res}}(L\mathfrak{sl}_3)$-module $L(m)$ can be written
\begin{equation*}
L(m)= L(\mathbf{Y}_{1}^k\mathbf{Y}_{2}^\ell)\otimes L(m^0),
\end{equation*}
where $L(m^0)$ is one of the following eight tensor products:
\begin{equation*}
\begin{array}{rllrl}
(i)&L(Y_{1,0})^{\otimes a}\otimes L(Y_{1,0}Y_{2,1})^{\otimes b} & & (ii)&L(Y_{2,1})^{\otimes a}\otimes L(Y_{1,0}Y_{2,1})^{\otimes b}     \\
(iii)&L(Y_{1,0})^{\otimes a}\otimes L(Y_{1,0}Y_{2,3})^{\otimes b} & &(iv)&L(Y_{2,3})^{\otimes a}\otimes L(Y_{1,0}Y_{2,3})^{\otimes b}      \\
(v)&L(Y_{1,2})^{\otimes a}\otimes L(Y_{1,2}Y_{2,1})^{\otimes b} &  & (vi)&L(Y_{2,1})^{\otimes a}\otimes L(Y_{1,2}Y_{2,1})^{\otimes b}    \\
(vii)&L(Y_{1,2})^{\otimes a}\otimes L(Y_{1,2}Y_{2,3})^{\otimes b} & & (viii)& L(Y_{2,3})^{\otimes a}\otimes L(Y_{1,2}Y_{2,3})^{\otimes b}  .    \\
\end{array}
\end{equation*}
\end{theorem}

\begin{proof}
 In this case, any $l$-acyclic dominant monomial $m^0$ is of the form   $$m^0=Y_{1,0}^{a_{10}}Y_{1,2}^{a_{12}}Y_{2,1}^{a_{21}}Y_{2,3}^{a_{23}},$$ with $a_{10}a_{12}=0$ and $a_{21}a_{23}=0$. The proof is the same for all eight situations, so we only check $(i)$.
 Case $(i)$ corresponds to $a_{12}=a_{23}=0$ and $a_{10}\geq a_{21}$. Then we have $m^0= (Y_{1,0})^{a_{10}-a_{21}} (Y_{1,0}Y_{2,1})^{a_{21}},$ and we deduce from Lemma \ref{caracA2} that 
$
L(m^0)=  L(Y_{1,0})^{\otimes a_{10}-a_{21}}\otimes L(Y_{1,0}Y_{2,1})^{\otimes a_{21}}.
$
\hfill $\Box$
\end{proof}


The modules $L(m)$ where $m$ is $l$-acyclic satisfy some interesting relations.

\begin{proposition}\label{relGc}
The following identities hold, for $i\in\{0,2\}$ and $j\in\{1,3\}$.
\begin{equation}
\begin{array}{ll}
\chi_\varepsilon(L(Y_{1,i}))\chi_\varepsilon(L(Y_{2,j}))&=\chi_\varepsilon(L(Y_{1,i}Y_{2,j})) +1 \\
\chi_\varepsilon(L(Y_{1,i}Y_{2,1}))\chi_\varepsilon(L(Y_{1,i}Y_{2,3}))&=  \chi_\varepsilon(L(Y_{1,i}))^3 + \chi_\varepsilon(L(Y_{1,i}))^2 \chi_\varepsilon(L(\mathbf{Y}_{2})) \\&\quad+ \chi_\varepsilon(L(Y_{1,i})) \chi_\varepsilon(L(\mathbf{Y}_{1})) + 1 \\
\chi_\varepsilon(L(Y_{1,0}Y_{2,j}))\chi_\varepsilon(L(Y_{1,2}Y_{2,j}))&= \chi_\varepsilon(L(Y_{2,j}))^3 + \chi_\varepsilon(L(Y_{2,j}))^2 \chi_\varepsilon(L(\mathbf{Y}_{1})) \\&\quad+ \chi_\varepsilon(L(Y_{2,j})) \chi_\varepsilon(L(\mathbf{Y}_{2})) + 1. 
\end{array}
\end{equation}
\end{proposition}

\begin{proof}
In addition to the expressions of \eqref{A22}, we have the following formulas:  
\begin{equation}
\begin{array}{ll}
\chi_\varepsilon(L(Y_{2,1})) &= Y_{2,1} + Y_{1,2}Y_{2,3}^{-1} + Y_{1,0}^{-1}\\
\chi_\varepsilon(L(\mathbf{Y}_{1}))& = Y_{1,0}Y_{1,2} + Y_{2,1}Y_{2,3}Y_{1,0}^{-1}Y_{1,2}^{-1}+ Y_{2,1}^{-1}Y_{2,3}^{-1}\\
\chi_\varepsilon(L(\mathbf{Y}_{2})) &= Y_{2,1}Y_{2,3}+ Y_{1,0}Y_{1,2}Y_{2,1}^{-1}Y_{2,3}^{-1} + Y_{1,0}^{-1}Y_{1,2}^{-1}.
\end{array}
\end{equation}
All the relations can then be obtained by straightforward computations. \hfill $\Box$
\end{proof}

Just like for $\mathfrak{sl}_2$, Section \ref{s223} implies that the Grothendieck ring $R:=K_0(\CZres)$ for $\mathfrak g = \mathfrak{sl}_3,\: l=2,$ is isomorphic to the polynomial ring:
\begin{equation}\label{isocarG}
R \cong \mathbb{Z} \lbrack  \lbrack L( Y_{1,0}) \rbrack, \lbrack L( Y_{1,2})\rbrack, \lbrack L( Y_{2,1})\rbrack, \lbrack L( Y_{2,3})\rbrack       \rbrack.
\end{equation}

Recall that for $\mathfrak g = \mathfrak{sl}_3$, the Grothendieck ring of the finite-dimensional representations of $\mathfrak g$ is isomorphic to the polynomial ring $\mathbb Z\lbrack \lbrack V(\varpi_1)\rbrack, \lbrack V(\varpi_2)\rbrack 
$. In this case, the simple modules are of the form $V(a_1\varpi_1+a_2\varpi_2)$ and can be written as polynomials in $\lbrack V(\varpi_1)\rbrack,\lbrack V(\varpi_2)\rbrack$, which creates a 2-parameter family of polynomials in 2 variables, denoted by $S_{a_1,a_2}(\lbrack V(\varpi_1)\rbrack,\lbrack V(\varpi_2)\rbrack)$. These polynomials can be computed inductively using the Littlewood-Richardson rule. The Frobenius pullback maps the class $\lbrack V(a_1\varpi_1+a_2\varpi_2)\rbrack $ to the class $\lbrack L(\mathbf{Y}_{1}^{a_1} \mathbf{Y}_{2}^{a_2})\rbrack$, which can then be written as the polynomial $S_{a_1,a_2}(\lbrack L(\mathbf{Y}_{1})\rbrack ,\lbrack L(\mathbf{Y}_{2})\rbrack)$.


Let $\overline{\mathcal{G}}$ be the generalised cluster algebra of type $G_2$, with $\mathbb{P}=\mathrm{Trop}(\lambda_1,\lambda_2)$, initial cluster variables $x_1,x_2$, exchange matrix 
\begin{equation*}B=\left( \begin{array}{cc}
0&3\\
-1&0
\end{array}    \right),    \end{equation*}
and initial exchange polynomials 
\begin{equation}\theta_1^0 (u,v)=u+v, \qquad \theta_{2}^0(u,v)= u^3 + \lambda_1 u^2 v + \lambda_2uv^2+ v^3.\end{equation}
There are eight cluster variables $x_1,\dots,x_8$, organised in eight clusters, as in Figure \ref{ex:g2}. Note that the exchange polynomials are not affected by mutation.
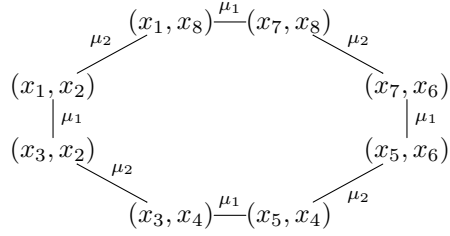
\begin{figure}[!ht]
\begin{equation*}\xymatrix  @R=1.2pc @C=1pc @M=0.0pc    {& (x_1,x_8)\ar@{-}[r]^{\mu_1} & (x_7,x_8)\ar@{-}[dr]^{\mu_2}&\\
(x_1,x_2)\ar@{-}[ur]^{\mu_2}\ar@{-}[d]^{\mu_1} &&&(x_7,x_6)\ar@{-}[d]^{\mu_1}\\
(x_3,x_2)\ar@{-}[dr]^{\mu_2}&&&(x_5,x_6)\ar@{-}[dl]^{\mu_2}\\
&(x_3,x_4)\ar@{-}[r]^{\mu_1} &(x_5,x_4)
}\end{equation*}
\caption{The clusters in type $G_2$}\label{ex:g2}
\end{figure}

Let $\mathcal G$ be the $\mathbb Z\lbrack \lambda_1,\lambda_2\rbrack$-subalgebra of the ambient field $\mathcal F$ generated by the cluster variables of $\overline{\mathcal{G}}$. 
As for $\mathcal{A}_n$, one can also check that $\mathcal G$ is isomorphic to a polynomial ring:
\begin{equation}\label{isoG}
\mathcal G \cong \mathbb Z \lbrack x_1,x_3,x_5,x_7\rbrack.
\end{equation}
Denote by $\mathcal M_0$ the set of generalised cluster monomials of $\mathcal G$
. Then the set
\begin{equation}
\mathcal H := \{ S_{a_1,a_2}(\lambda_1,\lambda_2)\cdot m,\:a_1,a_2\in\mathbb N,\:m\in\mathcal{M}_0\}
\end{equation}
is a $\mathbb Z$-basis of $\mathcal G$. 
%
Let us now exhibit a cluster structure on $R$.

\begin{theorem}\label{thmA2l2}
For $l=2$ and $\mathfrak g = \mathfrak{sl}_3$, there exists a ring isomorphism $\eta :\mathcal G \rightarrow R$ such that
\begin{equation}\label{isoGcorr}
\begin{array}{ll}
\eta(x_1)= \lbrack L( Y_{1,0}) \rbrack      & \eta(x_2)=  \lbrack L( Y_{1,0}Y_{2,3}) \rbrack       \\
\eta(x_3)= \lbrack L( Y_{2,3}) \rbrack      & \eta(x_4)= \lbrack L( Y_{1,2}Y_{2,3}) \rbrack        \\
\eta(x_5)= \lbrack L( Y_{1,2}) \rbrack      & \eta(x_6)= \lbrack L( Y_{1,2}Y_{2,1}) \rbrack         \\
\eta(x_7)= \lbrack L( Y_{2,1}) \rbrack     & \eta(x_8)= \lbrack L( Y_{1,0}Y_{2,1}) \rbrack        \\
\eta(\lambda_1)= \lbrack L(\mathbf{Y}_{1})\rbrack      & \eta(\lambda_2)=  \lbrack L(\mathbf{Y}_{2})\rbrack      \\
\end{array}
\end{equation}
The $\mathbb Z$-basis $\mathcal E\subset\mathcal G$ of monomials in $x_1,x_3,x_5,x_7$ is mapped by $\eta$ to the basis of classes of standard modules in $R$. Moreover, the $\mathbb Z$-basis $\mathcal H$ of generalised cluster monomials of $\mathcal G$ is mapped to the basis $B$ of classes of simple modules in $R$.
\end{theorem}

More precisely, the clusters can be organised as in Figure \ref{mod:g2}.

\begin{figure}[!ht]
\begin{center}
\small
$\xymatrix @R=1pc @C=0.8pc @M=0.0pc   { &\mbox{$\begin{array}{c}\lbrack L(Y_{1,0})\rbrack,\\\lbrack L(Y_{1,0}Y_{2,1})\rbrack\end{array} $} \ar@{-}[r]^{\mu_1}\ar@{-}[dl]^{\mu_2} & \mbox{$\begin{array}{c}\lbrack L(Y_{2,1})\rbrack,\\\lbrack L(Y_{1,0}Y_{2,1})\rbrack\end{array}$} \ar@{-}[dr]^{\mu_2} &\\
 \mbox{$\begin{array}{c}\lbrack L(Y_{1,0})\rbrack,\\\lbrack L(Y_{1,0}Y_{2,3})\rbrack\end{array}$}\ar@{-}[d]^{\mu_1} &&&\mbox{$\begin{array}{c}\lbrack L(Y_{2,1})\rbrack,\\\lbrack L(Y_{1,2}Y_{2,1})\rbrack\end{array}$}\ar@{-}[d]^{\mu_1} \\
 \mbox{$\begin{array}{c}\lbrack L(Y_{2,3})\rbrack,\\\lbrack L(Y_{1,0}Y_{2,3})\rbrack\end{array}$}\ar@{-}[dr]^{\mu_2}&&&\mbox{$\begin{array}{c}\lbrack L(Y_{1,2})\rbrack,\\\lbrack L(Y_{1,2}Y_{2,1})\rbrack\end{array}$}\ar@{-}[dl]^{\mu_2} \\
& \mbox{$\begin{array}{c}\lbrack L(Y_{2,3})\rbrack\\\lbrack L(Y_{1,2}Y_{2,3})\rbrack\end{array}$}\ar@{-}[r]^{\mu_1}&  \mbox{$\begin{array}{c}\lbrack L(Y_{1,2})\rbrack\\\lbrack L(Y_{1,2}Y_{2,3})\rbrack\end{array}$}
}$
\end{center}
\caption{The cluster structure for $A_2$, $l=2$}\label{mod:g2}
\end{figure}
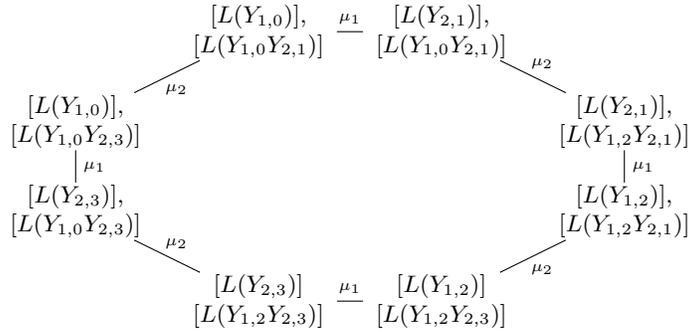
\normalsize


\begin{proof}
The proof is analogous to the $A_1$ case. The isomorphisms \eqref{isocarG} and \eqref{isoG} allow us to fix a ring isomorphism $\eta$, which sends each cluster variable $x_1,x_3,x_5,x_7$ to the class of the corresponding fundamental module as written above in \eqref{isoGcorr}. Moreover, Proposition \ref{relGc} and the definition of $\mathcal G$ imply that the cluster variables and their images under $\eta$ satisfy the same relations, where $\chi_\varepsilon(L(\mathbf{Y}_{i}))$ is replaced by $\lambda_i$. Therefore, we have $\eta(\lambda_i)=  \lbrack L(\mathbf{Y}_{i}) \rbrack$ for $i=1,2$, and the correspondence between the basis $\mathcal E$ and the basis of classes of standard module immediately follows.

Moreover, Theorem \ref{decG} and the description of the classes  $\lbrack L(\mathbf{Y}_{1}^{a_1} \mathbf{Y}_{2}^{a_2})\rbrack$ as the polynomials $S_{a_1,a_2}(\lbrack L(\mathbf{Y}_{1})\rbrack ,\lbrack L(\mathbf{Y}_{2})\rbrack)$, imply that the basis of classes of simple modules in $R$ consists of elements of the form \begin{equation}S_{a_1,a_2}(\lbrack L(\mathbf{Y}_{1})\rbrack ,\lbrack L(\mathbf{Y}_{2})\rbrack)\cdot M,\end{equation} where $a_1,a_2\in\mathbb N$ and $M$ is the class of one of the eight tensor products described in Theorem \ref{decG} (which corresponds to an element of the set $\mathcal M_0$ in $\mathcal G$). This allows us to conclude that the basis $\mathcal H$ of $\mathcal G$ is mapped to the basis $B$ of $K_0(\CZres)$. \hfill $\Box$

\end{proof}

\subsection{The case \texorpdfstring{$A_2$, $l>2$}{A2,l>2}}

Our various computations have led us to the following conjecture.

\begin{conjecture}\label{conjA2}
For $l\geq 2$, the Grothendieck ring of $\CZres$, is isomorphic to a generalised cluster algebra $\mathcal G_l$ of rank $2l-2$. An initial seed of $\mathcal G_l$ is given by the exchange matrix
\begin{equation}
\left( \begin{array}{cccccccccccccccccccccccccccccccccccccccccccccccccccccccccccccccccccccccccccccccccccccccccccccccccccccccc}
0&-1&1&0&\dots&0&\dots&0&    0&0&0\\
1&0&-1&1&0&0 &\dots&0& 0&0&0 \\
-1&1&0&-1&1&0 &\ddots& 0&0&0&0        \\ 
 0&-1&1&0&-1&1&0&\vdots&\vdots&\vdots&\vdots           \\
 
0&\ddots&\ddots&\ddots&\ddots&\ddots&\ddots&\ddots&     0&0&0\\
0&\dots&0&-1&1&0&-1&1&     0&0&0\\
0&\dots&&0&-1&1&0&-1&     1&0&0\\
&&&0&0&-1&1&0&     -1&1&0\\
\vdots&\ddots&&0&0&0&-1&1&     0&-2&3\\
&&&0&0&0&0&-1&     2&0&-3\\
0&&\dots&0&0&0&0&0&     -1&1&0
\end{array}  \right), 
\end{equation}
the cluster variables correspond to
\begin{eqnarray*}
&&x_{2k+1}= \lbrack L(Y_{1,0}Y_{1,2l-2}\dots Y_{1,2l-2k})\rbrack\quad(k\in\llbracket 0,l-2\rrbracket),\\&&x_{2k}=\lbrack L(Y_{2,2l-1}Y_{2,2l-3}\dots Y_{2,2l-2k+1})\rbrack\quad(k\in\llbracket 1,l-2\rrbracket),\\&&
x_{2l-2}=\lbrack L(Y_{1,0}Y_{1,2l-2}Y_{1,2l-4}\dots Y_{1,4}Y_{2,2l-1}Y_{2,2l-3}Y_{2,2l-5}\dots Y_{2,5}Y_{2,3})\rbrack.
\end{eqnarray*}
The coefficients are $\lambda_i=\lbrack L(\mathbf{Y}_{i})\rbrack,\:i=1,2,$ where \begin{equation*}
\mathbf{Y}_1= Y_{1,0}Y_{1,2}\dots Y_{1,2l-2}\quad\mbox{and}\quad \mathbf{Y}_2=Y_{2,1}Y_{2,3}\dots Y_{2,2l-1}.
\end{equation*} 

The initial exchange polynomials are $\theta_r^0 (u,v)=u+v$ for $r\in\llbracket 1,2l-3\rrbracket$, and $\theta_{2l-2}^0(u,v)= u^3 + \lambda_1 u^2 v + \lambda_2uv^2+ v^3$.

Moereover, the generalised cluster monomials are mapped to classes of simple modules.

\end{conjecture}


\begin{remark}
For $l>2$, the above generalised cluster algebras are of infinite type, so we can only hope for an inclusion of the set of cluster monomials in the set of classes of simple modules $L(m)$ where $m$ is $l$-acyclic.
\end{remark}

\renewcommand{\refname}{References}

\flushleft LMNO, CNRS UMR 6139, Université de Caen, 14032 Caen cedex, France\\
email: \url{anne-sophie.gleitz@unicaen.fr}

\end{document}